\documentclass[11pt]{amsart}
\usepackage{fullpage}
\usepackage[titletoc]{appendix}

\usepackage{amssymb,amsmath,amsthm,amscd,mathrsfs,graphicx, color}
\usepackage[cmtip,all,matrix,arrow,tips,curve]{xy}
\usepackage{hyperref}

\numberwithin{equation}{section}
\newtheorem{teo}{Theorem}[section]
\newtheorem{pro}[teo]{Proposition}
\newtheorem{lem}[teo]{Lemma}
\newtheorem{cor}[teo]{Corollary}

\newtheorem{que}[teo]{Question}

\newtheorem{teoalpha}{Theorem}

\theoremstyle{definition}
\newtheorem{dfn}[teo]{Definition}
\newtheorem{exa}[teo]{Example}

\theoremstyle{remark}
\newtheorem{rem}[teo]{Remark}

\DeclareMathOperator{\coniveau}{N}
\def\etconiveau{\coniveau_\ell}
\def\hconiveau{\coniveau_H}

\def\cross{\times}

\def\cx{{\mathbb C}}
\def\ff{{\mathbb F}}

\def\rat{{\mathbb Q}}
\def\integ{{\mathbb Z}}
\def\idp{{\mathfrak p}}
\def\calo{{\mathcal O}}

\def\iso{\cong}

\renewcommand{\bar}[1]{{\overline{#1}}}
\newcommand{\ubar}[1]{{\underline{#1}}}

\DeclareMathOperator{\alb}{Alb}

\DeclareMathOperator{\End}{End}
\DeclareMathOperator{\Hom}{Hom}

\DeclareMathOperator{\gal}{Gal}

\DeclareMathOperator{\spec}{Spec}
\DeclareMathOperator{\pic}{Pic}
\DeclareMathOperator{\Ab}{Ab}
\DeclareMathOperator{\A}{A}

\DeclareMathOperator{\chow}{CH}
\DeclareMathOperator{\Chow}{Chow}

\def\ra{\rightarrow}
\def\tensor{\otimes}
\newcommand{\st}[1]{\left\{#1\right\}}

\newcommand{\LKtrace}[1]{\underline{\underline{#1}}}

\newenvironment{alphabetize}{\begin{enumerate}

}{\end{enumerate}}

\begin{document}


\title[Descending cohomology]{On descending cohomology geometrically}

\author{Jeffrey D. Achter}
\address{Colorado State University, Department of Mathematics,
Fort Collins, CO 80523,
  USA}
  \email{j.achter@colostate.edu}

\author{Sebastian Casalaina-Martin }
\address{University of Colorado, Department of Mathematics, 
Boulder, CO 80309, USA }
\email{casa@math.colorado.edu}

\author{Charles Vial}
\address{University of Cambridge, DPMMS,  Cambridge CB3 0WB,
UK}
\email{c.vial@dpmms.cam.ac.uk}

\thanks{The first author was partially supported by  grants from the
  Simons Foundation (204164) and the NSA (H98230-14-1-0161,
H98230-15-1-0247 and H98230-16-1-0046).  The second author
was partially supported by NSF grant DMS-1101333 and a Simons Foundation
Collaboration Grant for Mathematicians
(317572).  The third author was supported by the Fund for Mathematics at the
Institute for Advanced Study and by EPSRC Early Career Fellowship EP/K005545/1.}

\date{October 26, 2016}

\begin{abstract}
 In this paper, motivated by a problem posed by Barry Mazur,   we show that for  smooth projective  varieties over the rationals, the   odd 
cohomology groups of degree less than  or equal to the dimension 
 can be
modeled by the cohomology of an abelian variety, 
provided the geometric coniveau is maximal.  This provides an affirmative  answer to Mazur's question for all uni-ruled 
threefolds, for instance.
Concerning cohomology in degree three, we show that the image of the Abel--Jacobi map
admits a distinguished model over the rationals. 
\end{abstract}

\maketitle


\section*{Introduction}

We consider the  problem of determining  when the cohomology of a
smooth 
projective variety $X$ over a field $K$ can be modeled  by an abelian variety.
For example, for a prime number $\ell$ different from the characteristic of $K$,
it is well known that the first $\ell$-adic cohomology group $H^1(X_{\bar
K},\rat_\ell)$ can be modeled by the Albanese variety $\operatorname{Alb}(X)$ in
the sense that it is isomorphic as a $\operatorname{Gal}(K)$-representation to
$H^1(\operatorname{Alb}(X)_{\bar K},\rat_\ell)$.
   Our primary 
motivation is a special case of a problem posed by Barry Mazur
\cite{mazurprob,mazurprobICCM} at the   birthday conference for Joe Harris in
2011. To fix ideas, 
suppose $X$ is a smooth projective variety  over
 a number field $K$.  For each prime number $\ell$ and nonnegative integer $n$, 
   the cohomology group
  $H^{2n+1}(X_{\bar K},\rat_\ell)$ admits a continuous action by
  $\gal(K)$, and $H^{2n+1}(X_{\bar K},\rat_\ell(n))$ has weight one.
 Now further suppose that $h^{p,q}=0$ for $p+q=2n+1$, $p,q\ne n$.   Thanks to Mazur's Newton over
Hodge 
  Theorem \cite{mazurfrobhodge}, 
  the twisted Galois representation
  $H^{2n+1}(X_{\bar K},\rat_\ell(n))$ is still effective\,; each eigenvalue
  of Frobenius is actually an algebraic integer.    Thus,
  $H^{2n+1}(X_{\bar K},\rat_\ell(n))$ could be the cohomology group of
  some abelian variety in the sense that there is no obvious obstruction to the
  existence of an abelian variety $J/K$ such that for all prime numbers $\ell$ there is an isomorphism of
$\operatorname{Gal}(K)$-representations 
\begin{equation}
\label{eqmodel}
H^1(J_{\bar K},\rat_\ell)\cong 
H^{2n+1}(X_{\bar K},\rat_\ell(n)).
\end{equation}
Mazur's question
  is to  determine  when such an abelian variety
$J/K$
 actually
exists.

Note that an isomorphism as in   
\eqref{eqmodel}
only characterizes the 
abelian variety $J/K$ up to isogeny.  This ambiguity persists even if one
works
with
$\integ_\ell$-cohomology groups. 
Indeed, there exist elliptic curves
$E_1$ and
$E_2$ over a number field such that, for each $\ell$, the Tate modules $T_\ell E_1$ and
$T_\ell E_2$ are isomorphic as Galois-representations\,; and yet $E_1$ and $E_2$ are
not isomorphic, even over $\cx$ \cite[Sec.\ 12]{zarhin08}.
Following Mazur, the {\em phantom
isogeny class} for $X$, in degree $n$, is the isogeny class of abelian varieties
satisfying \eqref{eqmodel}.  Any member of this isogeny class will be called a
{\em phantom abelian variety} for $X$.
 (We make the
analogous definition for varieties over an arbitrary  \nolinebreak[4] base  \nolinebreak[4] 
field.)

This type of question has of course been considered in many contexts,
and results in the literature provide a great deal of motivation.  To
begin, from the arithmetic perspective, 
the theorems of Tate and Honda
\cite{tate66,tate71, honda68}  imply 
 (under certain further hypotheses\,; see \S \ref{S:FinField}) that for
a prime $\idp$  of good
reduction for $X$, there exists an abelian variety $J_\idp/\kappa(\idp)$ such that, if a
phantom abelian variety $J/K$ were to exist, then the
reduction of $J$ modulo $\idp$ would be isogenous to $J_\idp$.  
Further 
motivation comes from complex geometry.
Under our hypotheses on the Hodge numbers,  
  the intermediate Jacobian
$J^{2n+1}(X_{\mathbb C})=H^{n,n+1}(X_{\mathbb C})/H^{2n+1}(X_{\mathbb
  C},\mathbb Z)$ is equal to the algebraic intermediate Jacobian,
and is consequently an abelian variety.   Since the dual abelian variety $\widehat
J^{2n+1}(X_{\mathbb C})$ satisfies $H^1(\widehat J^{2n+1}(X_{\mathbb C}),\mathbb
Z_\ell)=H^{2n+1}(X_{\mathbb C},\mathbb Z_\ell(n))$, 
  if the intermediate
Jacobian were isogenous to an abelian variety that  descended to $K$,
this would provide a candidate for a phantom abelian variety.
As justification for this approach, recall that  a theorem of  Deligne \cite{deligneniveau} implies
that for a complete intersection of odd dimension with Hodge level $1$,
 the intermediate
Jacobian descends to $K$
 and  its dual provides a phantom abelian
variety.

     In
this paper, we establish  the existence of phantom
 abelian varieties for more general smooth projective varieties, but  under some
further hypotheses on the cohomology.
Namely, we make  the stronger, yet
 natural, hypothesis that
$H^{2n+1}(X_{\mathbb C},\rat)$ has geometric coniveau
$n$, meaning that $H^{2n+1}(X_{\mathbb C},\rat)$ is supported on a subvariety of $X_\cx$ of codimension $n$.  In fact, assuming the Tate
conjecture 
for $\widehat J \times_K X$, the geometric coniveau $n$ part is an \emph{a priori} upper bound on
what can be modeled directly with an abelian variety.

\begin{teoalpha}[(Theorem \ref{T:cycles'})]\label{T:cycles}  
Suppose $X$ is a smooth 
projective variety  
of pure dimension 
 over a field 
$K\subseteq
\mathbb C$.  If  $H^{2n+1}(X_{\mathbb C},\mathbb Q)$ has geometric coniveau
$n$, for some integer $n$ with $2n+1\le \dim X$, then 
 there exist an abelian variety $J$ over $K$, with $J_{\mathbb C}$ isogenous to
the   intermediate Jacobian $J^{2n+1}(X_{\mathbb C})$, and a
correspondence on $\widehat J\cross _KX$ that  induces, for each prime $\ell$, an
isomorphism of  $\operatorname{Gal}(K)$-representations
\begin{equation}\label{E:Tcyc}
H^1(\widehat J_{\bar K},\rat_\ell)\cong H^{2n+1}(X_{\bar K},\rat_\ell(n)).
\end{equation}
Moreover, its inverse is induced by a correspondence on $X\cross _K\widehat J$.
\end{teoalpha}

We note that the theorem provides an affirmative answer to Mazur's question for
all threefolds with $\operatorname{CH}_0$ supported on a surface (e.g.,  for 
all uni-ruled threefolds\,; see Corollary \ref{C:cycles3}), and, assuming the generalized Hodge conjecture, for all threefolds with $h^{3,0}=0$. 
More generally, Theorem \ref{T:cycles'} provides for all $n$ a phantom for the geometric coniveau part $\coniveau^n H^{2n+1}(X_{\bar K},\rat_\ell(n))$ as soon as $X$ satisfies the Lefschetz standard conjecture.

In Theorem \ref{T:cycles}, operating only under the hypothesis of geometric
coniveau $n$, we find a suitable phantom $J/K$ and show that its dual  is
geometrically
isogenous to the intermediate Jacobian $J^{2n+1}(X_\cx)$.  This does not imply
that $J^{2n+1}(X_\cx)$ itself descends to $K$. 
 In contrast, in the special case $n=1$, we can 
 single out an isomorphism
class (in fact, a model over $K$ of the dual of the intermediate Jacobian)
within the phantom isogeny class. 
More generally, without any hypothesis on geometric coniveau, we extend Murre's work \cite{murre83} on algebraic representatives to the setting of non-algebraically closed fields to show (see Corollary \ref{C:cycles3})\,:

\begin{teoalpha}\label{T:jacdescent} 
 Let $X$ be a smooth 
 projective  variety over a field $K\subseteq
 \mathbb C$.  
 The  abelian variety  $J^3_{a}(X_{\mathbb C})$ that is the image of the Abel--Jacobi map $AJ : A^2(X_\cx)\to J^3(X_\cx)$ 
 has a distinguished  model $J$ over $K$ making the Abel--Jacobi map  $AJ : A^2(X_\cx)\twoheadrightarrow  J_a^3(X_\cx)$    $\operatorname{Aut}(\mathbb C/K)$-equivariant.
 Moreover, $\hat J$ provides a phantom for  $\coniveau^1H^3(X_{\bar K},\rat_\ell(1))$.
\end{teoalpha}

We also establish results regarding specialization  of algebraic representatives (\S \ref{S:Specialization}), as well as algebraically closed base change of algebraic representatives (Theorem \ref{T:ACBC}).  
As suggested by Lenny Taelman,
our   results extending Murre's work  on algebraic representatives 
 can be interpreted in a functorial framework to provide algebraic representatives over any perfect field.  More precisely, 
let $X$ be a smooth projective variety over a perfect field $K$. Denote by  $\operatorname{A}^2_\natural(X)$ the contravariant functor on the category of  smooth  integral varieties over $K$ to the category of abelian groups given by families of  codimension $2$ cycles on $X$ whose geometric  fibers are algebraically trivial. Then there exist an abelian variety $ \operatorname{Ab}^2(X )$ over $K$ and a natural transformation of contravariant functors $\operatorname{A}^2_\natural(X) \to  \operatorname{Ab}^2(X)$ that  is initial among all natural transformations of contravariant functors $\operatorname{A}^2_\natural(X) \to A$ to an abelian variety $A$ over $K$.   This approach will appear in  \cite{ACMVfunctor}. 

In concomitant work \cite{ACMVquadric}, following \cite {beauville77} we provide a geometric
construction of the phantom abelian variety when $X$ is an odd-dimensional quadric bundle over a surface over a
field of characteristic $\neq 2$.  
Mazur has also asked   for the relationship between phantom abelian varieties and Albaneses of Hilbert and Kontsevich spaces associated to  $X$.  In future work we will discuss how this question relates to  Theorem \ref{T:jacdescent}.

\noindent \textbf{Outline.}   
In Section \ref{S:Pr}, we review the notions of geometric, 
 Hodge,  and Tate coniveau filtrations and relate the final step of these filtrations on odd-degree cohomology groups with the cohomology of curves, the Abel--Jacobi map, and Mazur's question on phantoms. The case of finite fields is discussed in  \S
\ref{S:FinField}.
We prove  our results on phantom abelian varieties in two parts that complement each
other\,:  Theorem \ref{T:cycles}  (Theorem \ref{T:cycles'}) gives general criteria for the
existence of a phantom, and  Theorem \ref{T:jacdescent} (Theorem \ref{T:cycles3'} and Corollary \ref{C:cycles3}) shows that the image of the second Abel--Jacobi map descends to a field of definition.
The proof of Theorem \ref{T:cycles'} relies essentially on polarizations of Hodge structures, carried out in \S \ref{S:PrThA1}, while the proof of Theorem \ref{T:cycles3'} is based on Murre's theory of algebraic representatives \cite{murre83}, and is carried out in \S \ref{S:c2c}, \S \ref{S:Gal}, \S \ref{S:pThAII}. In Section \ref{S:c2c} we review the results of Murre on algebraic representatives for smooth projective varieties defined over algebraically closed fields and prove in Theorem \ref{T:ACBC} that algebraic representatives, if they exist, behave well under base change along algebraically closed fields. In Section \ref{S:Gal} we extend Murre's theory of algebraic representative to the setting of Galois field extensions and prove in Theorem \ref{T:Desc} that the algebraic representative, if it exists, inherits a Galois descent datum.    In Section  \ref{S:pThAII} we prove Theorem \ref{T:cycles3'}, providing the connection between algebraic representatives and phantoms.  In Section  \ref{S:Complements} we discuss specialization of algebraic representatives and phantom abelian varieties, as well as some results concerning algebraic representatives for higher codimension cycles.

\noindent \textbf{Acknowledgements.} 
We would like to thank Barry Mazur for his comments and encouragement. We are
very grateful to Claire Voisin for explaining the connection  between Mazur's
question and Murre's work, which directly led to our work on
algebraic representatives.  We thank the anonymous referees for a number of comments that have strengthened
our results, and improved the exposition.   
In addition, we
thank David
Grant,  Brendan Hassett, Bruno Kahn, Vasudevan  Srinivas, Lenny Taelman,
Rahbar Virk, Felipe Voloch, 
and Jonathan Wise for useful
conversations.

\noindent \textbf{Conventions.}
A \emph{variety} over a field is a geometrically reduced separated
scheme of finite type over that field.
A \emph{curve}
(resp.~\emph{surface})  is a variety of pure dimension $1$ (resp.~$2$).
Given a variety $X$ over a field,   $\operatorname{CH}^i(X)$  denotes  the Chow group of codimension~$i$ cycles modulo rational equivalence, and $\operatorname{A}^i(X)\subseteq \chow^i(X)$  denotes the subgroup of cycles algebraically equivalent to $0$.   
For an abelian variety $A$ over a field $K$, we denote by $\widehat A$ the dual
abelian variety $\operatorname{Pic}^0_{A/K}$.  

\section{Phantoms and coniveau filtrations}
\label{S:Pr}

The aim of this section is to highlight the relevance of the geometric coniveau filtration in the context of Mazur's question. The main question we will tackle in this work is Question \ref{Q:NivMod}, where we ask whether the deepest part of the geometric coniveau filtration on odd-degree cohomology groups can be modeled by abelian varieties. To this end we discuss Hodge, Tate, and geometric coniveau filtrations, and their link to the Abel--Jacobi map.
One of the key points for later arguments is  Proposition \ref{P:niveau} stating
that the deepest part of the geometric coniveau filtration can be realized via
the cohomology of a curve.

\subsection{Hodge theory and geometric Galois
  representations} \label{S:PrHtGr} 
 While our motivating
problem is posed purely in terms of Galois representations, it has a
natural Hodge-theoretic analogue.

\begin{que}[{\cite[Que.~2.43]{voisinDiag}}]\label{Qvoisinghc}
Let $X$ be a smooth complex projective variety.  
Given a weight~$i$ Hodge sub-structure $L \subseteq  H^i (X, \mathbb Q)$ of
Hodge coniveau~$r$, so that
$$
L_{\mathbb C} =L^{i-r,r}\oplus  \cdots \oplus L^{r,i-r},
$$
(and $L(r)$ is an effective  weight $i-2r$  Hodge structure),  
does there exist a smooth complex projective variety $Y$ admitting   a 
Hodge
sub-structure of $L'\subseteq H^{i-2r}(Y,\mathbb  Q)$ with $L'\cong L(r)$\,?
\end{que}

In the case that $i=2n+1$ is odd, and $L$ has Hodge coniveau $n$,  the
question has an affirmative answer.  Indeed, 
 recall that
the algebraic intermediate Jacobian, $J^{2n+1}_{\operatorname{alg}}(X)
\subseteq J^{2n+1}(X)$, is defined to be the largest complex sub-torus
with tangent space contained in $H^{n,n+1}(X)$. 
 This is an abelian
variety (e.g., \cite[p.304]{voisinI}), and clearly $L(n)$
 can be obtained as a
Hodge sub-structure of $H^1(\widehat J^{2n+1}_{\operatorname{alg}},\mathbb Q)$.

For Mazur's question we are interested in  an $\ell$-adic analogue of Question
\ref{Qvoisinghc}.  
Now let $K$ be a field, such as a number field, which is finitely generated over the prime field.
 A $\operatorname{Gal}(K)$-representation $V_\ell$
over $\mathbb Q_\ell$  has weight $i$ if, for each unramified prime $v$ of $K$
(and each embedding $\bar\rat_\ell \ra \cx$), each eigenvalue of Frobenius at
$v$ has size $\sqrt{q_v^i}$, where $q_v$ is the cardinality of the residue field
at $v$.
We say $V_\ell$ is entire  if
all eigenvalues of Frobenius are algebraic integers.  We say  $V_\ell$ is 
effective of weight $w$ if  $V_\ell$ and
$\operatorname{Hom}(V_\ell,\rat_\ell(-w))$ are
entire and of weight $w$.  
For a smooth projective variety $X$ over
$K$, by Poincar\'e
duality and hard Lefschetz, we have  $\operatorname{Hom}(H^i(X_{\bar K},\mathbb
Q_\ell),\mathbb Q_\ell(-i))\cong H^i(X_{\bar K},\mathbb Q_\ell)$, so that a
geometric representation $H^i(X_{\bar K},\mathbb Q_\ell(\nu))$ is effective (of
weight $i-2\nu$) if and only if it is entire (e.g.,
\cite[p.269]{jannsenseattle}).

\begin{que}
\label{Qgtc}
Let $X$ be a smooth
    projective variety over  a field $K$ that is finitely generated over the prime
    field.
Given a  $\operatorname{Gal}(K)$-sub-representation $L_\ell \subseteq  H^i
(X_{\bar{K}}, \mathbb Q_\ell)$  with the property that $L_\ell(r)$ is effective
(of weight $i-2r$), does there exist a smooth projective variety $Y$  over $K$ 
admitting a $\operatorname{Gal}(K)$-sub-representation  $L'_\ell\subseteq
H^{i-2r}(Y_{\bar K},\mathbb  Q_\ell)$ with $L'_\ell \cong L_\ell (r)$\,?
\end{que}

Unlike  the case of Hodge structures, even when $i=2n+1$ is odd, and $L_\ell(n)$
is effective of weight $1$, it is not known exactly when Question \ref{Qgtc} 
will have an affirmative answer.  Theorem \ref{T:cycles} provides an affirmative
 answer under  further hypotheses on the geometric coniveau\,;   we review
coniveau in the sections  below.   (See \S \ref{S:FinField} for a
discussion of the connection between  Question \ref{Qgtc} over finite
fields and the theory of Honda--Tate\,; see  \cite{ACMVabtriv} for a
discussion of the extent to which $Y$ may be taken to be geometrically
integral.)

\subsection{Coniveau filtrations and the generalized Hodge and Tate
conjectures}\label{S:PrCoNi}
Questions \ref{Qvoisinghc} and \ref{Qgtc} are related to the generalized Hodge
and Tate
 conjectures (\cite{grothBrauerIII}).    Our treatment here follows \cite[\S
3]{jannsenseattle}.
Let $X$ be a smooth projective
 variety over a  field   
$K$.
The \emph{geometric
coniveau filtration} $\coniveau^\nu H^i(X_{\bar{K}},\rat_\ell) $    is 
defined by\,: 
\begin{align}
\coniveau^\nu H^i(X_{\bar{K}},\rat_\ell)  
\label{eqconiveausub}
&= \sum_{\substack{ Z\subseteq
    X\\\text{closed, codim }\ge \nu}} \operatorname{Ker}(
H^i(X_{\bar{K}},\rat_\ell) \rightarrow
H^i(X_{\bar{K}}
 \backslash Z_{\bar{K}}, \rat_\ell)).
\end{align}
By resolution of singularities and by Deligne's theory of mixed Hodge structures when $\operatorname{char}\, K= 0$, or by de Jong's theory of alterations and by the Weil conjectures when $K$ is a finite field, the geometric coniveau filtration can also be characterized by the following formula (see \cite[Sec.\
4.4(d)]{illusiemiscellany})\,:
\begin{align}
\coniveau^\nu H^i(X_{\bar{K}},\rat_\ell)  
\label{eqconiveaumap}
&= \sum_{\substack{f:Y \rightarrow X\\\dim Y = \dim X - \nu \\
		Y\text{ smooth proj}}}\operatorname{Im}(f_*:H^{i-2\nu}(Y_{\bar{K}},
\rat_\ell(-\nu)) \rightarrow
H^i(X_{\bar{K}},\rat_\ell)).
\end{align}

Now suppose that $K$ is a sub-field of $\cx$.
The \emph{geometric coniveau filtration on Betti cohomology} 
$\coniveau^\nu H^i(X_\cx,\rat)$ is defined in  the  same way.
 Note that usually a class  $\alpha \in  H^i(X_\cx,\rat)$ is said to have
geometric coniveau $\nu$ if it vanishes when restricted to the complement of a
closed \emph{complex} subvariety of $X_\cx$ of pure codimension $\nu$. However,
using a standard
 argument 
 involving spreading out,
 specialization, and then Galois
 averaging, it suffices as indicated above to consider projective varieties $Z$, and  smooth projective 
 varieties $Y$, of pure dimension $\dim X-\nu$ defined
over
 $K$.

By virtue of the fact that  the canonical comparison isomorphisms between Betti,
de Rham and \'etale cohomologies are compatible with Gysin maps, and
even cycle maps \cite[\S 1]{dmos}, there are
canonical   isomorphisms
\begin{equation}
\label{E:CompIsom}
(\coniveau^\nu(H^i(X_\cx,\rat)))\tensor_\rat \rat_\ell \iso \coniveau^\nu
H^i(X_{\bar K},\rat_\ell),
\end{equation}
which justify our apparent abuse of notation.
The generalized Hodge and Tate conjectures relate these geometric coniveau
filtrations to two other filtrations.  

Call a $\mathbb Q$-Hodge structure $V$ effective if $V^{p,q}= 0$ unless $p,q\ge
0$. The \emph{Hodge coniveau filtration} $\hconiveau^\nu$ of  $H^i(X_\cx,\rat)$
has
term
$
\hconiveau^\nu H^i(X_\cx,\rat)
$
 defined to be the span
  of all Hodge sub-structures $V \subseteq
H^i(X_\cx,\rat)$
such that $V(\nu)$ is effective, i.e., such that $V^{p,q}=0$ unless $p,q \ge
\nu$.  The Gysin map provides a natural inclusion 
\begin{equation}
\label{eqghc}
\coniveau^\nu H^i(X_\cx,\rat) \subseteq \hconiveau^\nu H^i(X_\cx,\rat)\,;
\end{equation}
the \emph{generalized Hodge conjecture} (for $X$, in degree $i$) asserts
equality in \eqref{eqghc}.

Similarly, if $K$ is finitely generated over the prime field,
define a filtration
$
\etconiveau^\nu H^i(X_{\bar {K}}, \rat_\ell)
$
as the span 
of all sub-$\gal(K)$-representations  $V_\ell\subset H^i( X_{\bar
{K}},\rat_\ell)$ such that $V_\ell (\nu)$   is effective (of weight $i-2\nu$). 
Again, the Gysin map  (this time, for \'etale cohomology) provides a natural
inclusion
\begin{equation}
\label{eqgtc}
\coniveau^\nu H^i(X_{\bar K},\rat_\ell) \subseteq \etconiveau^\nu H^i(X_{\bar
K},\rat_\ell),
\end{equation}
and the \emph{generalized Tate conjecture}  (for $X$, in degree $i$) 
asserts equality in \eqref{eqgtc}.
Sometimes, we will abuse notation slightly and denote the $r$-th
twist of step $\nu$ in the geometric  coniveau filtration by
\[
\coniveau^\nu H^i(X_{\bar K},\rat_\ell(r)) := (\coniveau^\nu H^i(X_{\bar
  K},\rat_\ell))\tensor \rat_\ell(r)\,;
\]
we employ similar notation for the other filtrations (also on Betti cohomology).

Important to us will be the following characterization of the final step of the
geometric coniveau filtration for odd-degree cohomology groups\,:

\begin{pro}[{\cite[\S 1.2.3]{Vial2}}]\label{P:niveau}
Suppose $X$ is a smooth  projective variety  
over a field $K$, which is either of characteristic zero or finite,
and let $\ell\not = \operatorname{char}(K)$ be a prime. 
Let $n$ be a non-negative integer. Then there exist a smooth
projective curve
$C$ over $K$
and a correspondence $\gamma \in
\operatorname{CH}^{n+1}(C\times_K
X)\otimes_{\mathbb Z} {\mathbb{Q}}$  such that the induced morphism of
$\operatorname{Gal}(K)$-representations
$$
\gamma_*: H^1(C_{\overline K},{\mathbb{Q}}_\ell) \rightarrow
H^{2n+1}(X_{\overline K},{\mathbb{Q}}_\ell(n))
$$ 
has image
$\coniveau^nH^{2n+1}(X_{\overline K},{\mathbb{Q}}_\ell(n))$. 
If $K \subseteq \mathbb C$, then the morphism of Hodge structures
$$
\gamma_*: H^1(C_{\mathbb{C}},{\mathbb{Q}}) \rightarrow
H^{2n+1}(X_{\mathbb{C}},{\mathbb{Q}}(n))
$$ 
has image
$\coniveau^nH^{2n+1}(X_{\mathbb{C}},{\mathbb{Q}}(n))$.
Moreover, $\coniveau^n
H^{2n+1}(X_{\bar K},\rat_\ell(n))$
is a semisimple representation of $\gal(K)$.

\end{pro}

\begin{proof}  This is proved in \cite{Vial2} in the case of Betti 
cohomology (see the formula 
$\widetilde{\coniveau}^{\lfloor i/2 \rfloor}H_i(X) = {\coniveau}^{\lfloor i/2 \rfloor}H_i(X)$ in \S 1.2.3 therein). 
The proof adapts to our more general setting\,:
up to working component-wise, we may and do assume
that $X$ is connected, say of dimension $d$.
Since $H^{2n+1}(X_{\bar K},{\mathbb{Q}}_\ell(n))$ 
 is a finite-dimensional $\rat_\ell$-vector space,  from the characterization of coniveau  given in \eqref{eqconiveaumap} 
there exist 
  a 
 smooth projective variety  $Y$ of dimension $d_Y=d-n$ over $K$
and a
$K$-morphism $f:Y\to X$ such that 
$$ \coniveau^nH^{2n+1}(X_{\bar K},{\mathbb{Q}}_\ell(n)) = \mathrm{Im}\left(
f_*:H^1({Y}_{\bar K},{\mathbb{Q}}_\ell)
\rightarrow
H^{2n+1}(X_{\bar K},{\mathbb{Q}}_\ell(n))\right).
$$ 
  By Bertini \cite{poonen}, let $\iota : C \hookrightarrow Y$ be a one-dimensional smooth linear section of $Y$.   The hard Lefschetz 
theorem \cite[Thm.\ 4.1.1]{deligneweil2} states that intersecting with $C$ yields an isomorphism 
$$
\iota_*\iota^*:H^1(Y_{\bar K},\mathbb Q_\ell) \hookrightarrow H^1(C_{\bar K},\mathbb
Q_\ell)  \twoheadrightarrow H^{2d_Y-1}(Y_{\bar
K},\mathbb Q_\ell(d_Y-1)).
$$
The Lefschetz Standard Conjecture is known for $\ell$-adic cohomology and for
Betti cohomology in degree $\leq 1$ (see \cite[Thm.~2A9(5)]{kleiman}),
meaning in our case that the map
$\left(\iota_*\iota^*\right)^{-1}$ 
is induced by a correspondence $\Lambda \in \operatorname{CH}^1(Y\times_K
Y)_\rat$.   
Therefore, the composition 
\begin{equation}\label{E:LCoNi}
\xymatrix@C=2em{
 H^1(C_{\bar K},\mathbb Q_\ell) \ar@{->>}[r]^<>(0.5){\iota_*} &  
 H^{2d_Y-1}(Y_{\bar K},\mathbb Q_\ell(d_Y-1)) \ar@{->}[r]^<>(0.5){
  \Lambda_*}_<>(0.5)\cong & H^1(Y_{\bar K},\mathbb Q_\ell)
 \ar[r]^<>(0.5){f_*} & H^{2n+1}(X_{\bar K},\mathbb Q_\ell(n)),\\
}
\end{equation}
which has image
$\coniveau^nH^{2n+1}(X_{\bar K},{\mathbb{Q}}_\ell(n))$, is induced by the required correspondence.

That $\coniveau^n
H^{2n+1}(X_{\bar K},\rat_\ell(n))$
is a semisimple representation of $\gal(K)$ follows immediately from the fact that $H^1(C_{\bar
 K},\rat_\ell)$ is semisimple if $K$ is a field which is
finitely generated (over  $\rat$ \cite{faltings83,faltingswustholz} or a finite
field \cite{tate66,zarhin75}).

When $K \subseteq \cx$, the comparison isomorphisms \eqref{E:CompIsom} 
establish
that  the image  of the induced morphism of Hodge structures $\gamma_*:
H^1(C_{\cx},{\mathbb{Q}})
\rightarrow
H^{2n+1}(X_{\cx},{\mathbb{Q}}(n))$ is
$\coniveau^nH^{2n+1}(X_{\cx},{\mathbb{Q}}(n))$.   
\end{proof}

\subsection{Coniveau filtrations and the Abel--Jacobi map} \label{S:PrCoNiAJ}
In this section we review the connection between the coniveau filtrations and the
Abel--Jacobi map. 
Recall that, for a smooth complex projective variety $X$, the intermediate Jacobian is a complex torus defined by
\[
J^{2n+1}(X) := F^{n+1} H^{2n+1}(X,\cx) \backslash H^{2n+1}(X,\cx) /
H^{2n+1}(X,\integ),
\]
and that there is an \emph{Abel--Jacobi map} $$\mathrm{Ker}\left( \operatorname{CH}^n(X) \to H^{2n}(X,\mathbb Z)\right) \to  J^{2n+1}(X_{\mathbb C}).$$
By definition
the algebraic intermediate Jacobian $J^{2n+1}_{\operatorname{alg}}(X)$ is the largest sub-torus of $J^{2n+1}(X)$ that is projective\,; it
determines the $n$-th piece of the Hodge coniveau filtration\,:
\begin{eqnarray*}
\hconiveau ^nH^{2n+1}(X,\mathbb Q)\otimes_{\mathbb Q}\mathbb
C&=&T_0J^{2n+1}_{\operatorname{alg}}(X) \oplus \overline{T_0
J^{2n+1}_{\operatorname{alg}}(X)}\\
\hconiveau^nH^{2n+1}(X,\mathbb
Q(n))&=&H^1(\widehat J^{2n+1}_{\operatorname{alg}}(X),\mathbb Q).
\end{eqnarray*}

On the other hand, it is standard \cite[Th. 12.17]{voisinI}
that the $n$-th piece of the geometric coniveau filtration
$\coniveau^nH^{2n+1}(X,\mathbb Q)$ is determined by the image of the 
Abel--Jacobi map.
More precisely, the  image $J^{2n+1}_a(X)$ of the Abel--Jacobi map
$AJ:\operatorname{A}^{n+1}(X)\to J^{2n+1}(X)$ 
has tangent space $T_0J^{2n+1}_a(X)=H^{n,n+1}(X)\cap
\coniveau^nH^{2n+1}(X,\mathbb C)$, and 
\begin{eqnarray*}
 \coniveau^nH^{2n+1}(X,\mathbb Q)\otimes_{\mathbb Q}\mathbb
C&=&T_0J^{2n+1}_{a}(X)
\oplus \overline{T_0 J^{2n+1}_{a}(X)}\\
  \coniveau^nH^{2n+1}(X,\mathbb Q(n))&=&H^1(\widehat J^{2n+1}_{a}(X),\mathbb Q).
\end{eqnarray*}
Thus the Abel--Jacobi
map restricted to algebraically trivial cycles
$$AJ:\operatorname{A}^{n+1}(X)\to J^{2n+1}(X)$$ is surjective if and only if
$H^{2n+1}(X,\mathbb Q)$ has geometric coniveau $n$.

In light of this discussion,
the first goal of this paper is to consider analogous questions for geometric
Galois representations\,:

\begin{que} \label{Q:NivMod}
Given a smooth  
projective
 variety $X$ over a field $K$ finitely generated over the prime field, do there exist 
abelian varieties $\tilde J$ and $J$ defined over $K$ such that for each prime
$\ell$, 
\begin{eqnarray*}
\etconiveau^n H^{2n+1}(X_{\overline K},\mathbb Q_\ell(n)) &\cong &
H^1(\tilde
J_{\overline K},\mathbb Q_\ell)\\
 \coniveau^n H^{2n+1}(X_{\overline K},\mathbb Q_\ell(n)) &\cong& H^1(
J_{\overline
K},\mathbb Q_\ell)
\end{eqnarray*}
 as $\operatorname{Gal}(K)$-modules\,?  
\end{que}

The abelian variety $\tilde J/K$ would provide an answer to the
special case of Question \ref{Qgtc} that we are considering in this
paper, and consequently to Mazur's question (i.e., the case where $\etconiveau^n
H^{2n+1}(X_{\overline K},\mathbb Q_\ell(n)) = H^{2n+1}(X_{\overline K},\mathbb
Q_\ell(n))$). Note that if $\tilde J$ exists, then a direct consequence of the Tate conjecture for $X\times_K \tilde J$ would be that $\etconiveau^n$ and $\coniveau^n$ agree on $H^{2n+1}(X_{\overline K},\mathbb Q_\ell(n))$ (see also  \cite[p.34]{voisinDiag} for the corresponding Hodge-theoretic statement).
 Our aim in this paper is to
construct the abelian
variety $J/K$, under some further hypotheses.

\subsection{Mazur's question over finite fields}
\label{S:FinField}

Honda--Tate theory suggests an approach to Mazur's question over
finite fields.  Consider a smooth projective
variety  $X/\ff_q$
such that $h^{2n+1-j,j}(X)=0$ unless $j \in \st{n,n+1}$.  By Mazur's ``Newton
over Hodge'' theorem \cite{mazurfrobhodge}, each eigenvalue of Frobenius, acting
on $H^{2n+1}(X_{\bar\ff_q},\rat_\ell)$, is divisible by $q^n$, and thus
$H^{2n+1}(X_{\bar\ff_q},\rat_\ell(n))$ is an effective
$\gal(\ff_q)$-representation of weight one.  
It is then natural to ask if there is an abelian variety $J/\mathbb F_q$ such
that $H^1(J_{\bar\ff_q},\rat_\ell)\cong H^{2n+1}(X_{\bar\ff_q},\rat_\ell(n))$ as
Galois representations.  (More generally, of course, one has Question
\ref{Qgtc}.)  
Honda--Tate theory asserts a bijection
\begin{equation*}
\xymatrix{
\{\text{conjugacy classes of } q\text{-Weil numbers}\} \ar[r]^<>(0.5){\sim} & 
\{\text{isogeny classes of simple abelian varieties}/\mathbb F_q\}
}
\end{equation*}
which we crudely denote $\varpi \mapsto A_\varpi$.   The Galois
representation attached to an abelian variety over a finite field is
semisimple \cite{tate66}, and one immediately deduces from the above bijection that if
$V_\ell$ is a semisimple  effective weight one $\rat_\ell$-representation of
$\gal(\ff_q)$, then there exist an abelian variety $A/\ff_q$ and an
inclusion of $\gal(\ff_q)$-representations 
\begin{equation}
\label{eq:htincl}
\xymatrix{
V_\ell \ \ar@{^(->}[r] & H^1(A_{\bar\ff_q},\rat_\ell).
}
\end{equation}
By starting with a Weil number $\varpi$ for which
$[\rat(\varpi):\rat] < 2\dim A_\varpi$
    (see, e.g.,
\cite[p.\ 98]{tate71} for examples), it is easy to construct
semisimple representations $V_\ell$ for which \eqref{eq:htincl} is never an
isomorphism.  Nonetheless, we have\,:

\begin{lem}
  Let $X/\ff_q$ be a smooth projective variety.  Suppose that $X$ is
  of geometric coniveau $n$ in degree $2n+1$. 
\begin{alphabetize}
\item There is an abelian
  variety $A/\ff_q$ such that, for all $\ell\not =
  \operatorname{char}(\ff_q)$,
 there is an
  inclusion of Galois representations
\begin{equation}
\label{eq:phantomff}
\xymatrix{
H^{2n+1}(X_{\bar\ff_q},\rat_\ell(n)) \ \ar@{^(->}[r] & H^1(A_{\bar\ff_q},\rat_\ell).
}
\end{equation}
\item If $X$ is ordinary (the  Newton and Hodge polygons in degree $2n+1$ coincide), or if the Tate conjecture holds for
  $X\cross_{\ff_q}X$, then one may choose $A$ so that
  \eqref{eq:phantomff} is an isomorphism.
\end{alphabetize}
\end{lem}

\begin{proof}
Part (a) follows immediately from Honda--Tate theory and from the
semisimplicity of the Galois representation $H^{2n+1}(X_{\bar\ff_q},\rat_\ell(n))$ (see Proposition \ref{P:niveau}).

For part (b), first suppose that $X$ is ordinary in degree $2n+1$.
 It suffices to
verify that for each Weil number $\varpi$ occurring as an eigenvalue
of Frobenius in $H^{2n+1}(X_{\bar\ff_q},\rat_\ell(n))$, one has
$[\rat(\varpi):\rat] = 2 \dim A_\varpi$.  By ordinarity, all slopes of
the Newton polygon of $H^{2n+1}(X_{\bar\ff_q},\rat_\ell(n))$ are $0$
and $1$.  For each $\varpi$, the calculation of $\End(A_\varpi)$ in \cite[Thm.\
1]{tate71} implies that $[\rat(\varpi):\rat] = 2\dim A_\varpi$.

Second, suppose instead that the Tate conjecture holds for $X\cross_{\mathbb F_q}
X$.  Then one can calculate \cite[Prop.\
2.4
and 2.8]{milnemotivesff}
 the
dimension of (the motive corresponding to)
$H^{2n+1}(X_{\bar\ff_q},\rat_\ell(n))$.  The calculation of this dimension, which relies only on
the Weil numbers, is the same expression used to calculate the
dimension of an abelian variety with these Weil numbers.  In
particular, \eqref{eq:phantomff} is an isomorphism.
\end{proof}

\begin{rem}
In a similar spirit (but necessarily more involved fashion), Volkov
\cite{volkov05} has characterized the potentially crystalline $p$-adic
representations of
$\gal(\rat_p)$ which arise as summands of the $p$-adic Tate module of
an abelian variety over $\rat_p$.  As with finite fields, this lets
one conclude that if $X/\rat_p$ is a smooth  projective
threefold with
$h^{3,0}(X) = 0$, and if $X$ acquires good reduction over a tamely
ramified extension of $\rat_p$, then $H^3(X_{\bar\rat_p},\rat_p(1))$
is isomorphic to a sub-representation of $H^1(A_{\bar\rat_p},\rat_p)$
for some abelian variety $A/\rat_p$.
\end{rem}

\section{Phantoms via correspondences and polarizations} \label{S:PrThA1}

In this section we prove the following theorem, which implies both the first
part of
Theorem \ref{T:cycles}, and a partial, affirmative answer to Question
\ref{Q:NivMod}.   We will write $\operatorname{CH}^*(-)_\rat$ for
$\operatorname{CH}^*(-)
\otimes_{\mathbb Z}\rat$.

\begin{teo}\label{T:cycles'}  
Suppose $X$ is a smooth projective variety of pure dimension 
$d$ over a field  $K \subseteq \mathbb C$. 
Let $n$ be a non-negative integer.   Then there exist a
smooth (possibly reducible) 
projective curve $C$ over $K$,
 and a correspondence $\gamma \in
\operatorname{CH}^{n+1}(C\times_K
X)_{\mathbb{Q}}$,
such that for all
prime numbers $\ell$ the induced map
$$
\gamma_*:H^1((J_{C/K})_{\overline K},\mathbb Q_\ell) \to H^{2n+1}(X_{\overline
K},\mathbb Q_\ell(n))
$$
has image $\coniveau^nH^{2n+1}(X_{\overline K},\mathbb Q_\ell(n))$.    Assume
further any of the
following\,:
\begin{alphabetize}
\item $2n+1\leq 3$\,;

\item $X$ satisfies the Lefschetz standard conjecture\,;

\item $2n+1\leq d$ and $H^{2n+1}(X_{\mathbb C},\mathbb Q(n))$
 is of geometric
niveau $n-1$, in the sense of \cite{Vial2},
 \emph{i.e.},
 there exists a (possibly reducible) smooth
projective $3$-fold $Y$ over $K$
and a
correspondence $\Gamma \in
\operatorname{CH}^{n+2}(Y \times_K X)_{\rat}$ such that $\Gamma_* : H^3(Y_\cx,
\rat(1)) \rightarrow H^{2n+1}(X_\cx,\rat(n))$ is surjective\,;
\item $2n+1\leq d$ and $H^{2n+1}(X_{\mathbb C},\mathbb Q(n))$ has geometric
coniveau $n$.

\end{alphabetize}
Then there exists a sub-abelian variety $J\stackrel{\iota}{\hookrightarrow}
J_{C/K}$ defined  over $K$ 
 such that  for 
 all prime numbers  $\ell$, the correspondence $\gamma$ composed with the graph
of the inclusion $\Gamma_\iota$ induces  a split inclusion
of $\operatorname{Gal}(K)$-representations  
\begin{equation}\label{E:cycles'}
\xymatrix{
H^1(J_{\bar K},\mathbb Q_\ell) \ar@{^(->}[rr]^{\gamma_*\circ \Gamma_{\iota
*} \hspace{15pt}}& & H^{2n+1}(X_{\bar
K},\rat_\ell(n)),
}
\end{equation}
with the splitting induced by a correspondence over $K$ and with image $\coniveau^nH^{2n+1}(X_{\bar K},\rat_\ell(n))$.
In particular, if $2n+1\leq d$ and if $H^{2n+1}(X_{\mathbb C},\mathbb Q(n))$ is
of geometric
coniveau
$n$, then \eqref{E:cycles'} is an isomorphism.
\end{teo}

\begin{rem}
Let $X/K$ be a smooth integral projective variety.  Let $L/K$ be a finite
extension such
that $X_L$ is a
disjoint union of geometrically irreducible components, and
$\gal(L/K)$ acts transitively on these components.  If $\coniveau^nH^{2n+1}(-,\rat_\ell(n))$ of some
component of $X_L$ admits a phantom abelian variety $A$ over $L$,
then $\coniveau^nH^{2n+1}(X_{\bar{K}},\rat_\ell(n))$ admits the Weil
restriction ${\mathbf R}_{L/K}(A)$ of $A$ over $K$ as a phantom, since
$H^1({\mathbf R}_{L/K}(A)_{\bar K},\rat_\ell) \iso
\operatorname{Ind}_{\gal(L)}^{\gal(K)} H^1(A_{\bar K},\rat_\ell)$.
\end{rem}

\begin{proof}[Proof of Theorem \ref{T:cycles'}]
Proposition \ref{P:niveau} provides  
a smooth projective curve $C/K$ 
 and a correspondence $\gamma \in
\operatorname{CH}^{n+1}(C\times_K
X)_{\mathbb{Q}}$  such that the induced morphism of Hodge structures
$$
\gamma_*: H^1(C_{\mathbb{C}},{\mathbb{Q}}) \rightarrow
H^{2n+1}(X_{\mathbb{C}},{\mathbb{Q}}(n))
$$ 
has image
$\coniveau^nH^{2n+1}(X_{\mathbb{C}},{\mathbb{Q}}(n))$. This
establishes the first claim
of Theorem \ref{T:cycles'}.

In order to prove the remaining claims of Theorem \ref{T:cycles'}, we would
like   to find an abelian subvariety $J/K\hookrightarrow J_{C/K}$ so that the
composition 
$$
\xymatrix{
H^1(J_{\bar K},\mathbb Q_\ell) \ar@{^(->}[r]& H^1((J_{C/K})_{\bar K},\mathbb
Q_\ell)
\ar@{->>}[r]^<>(.5){\gamma_*} & \coniveau^nH^{2n+1}(X_{\overline K},\mathbb
Q_\ell(n))
}
$$
is an isomorphism.   At the level of rational Hodge structures and  complex
abelian varieties, this is elementary\,; the issue is to find a suitable abelian
variety
defined over $K$.    Rather than trying to directly descend an abelian variety
from $\mathbb C$ to $K$, our strategy for obtaining $J/K$ is motivic in nature
and consists in finding a
correspondence $\pi \in \operatorname{CH}^1(C\times_K C)_\rat$ which acts idempotently on $H^1(C_{\cx},\mathbb Q)$ such that 
$\pi_*H^1(C_{\cx},\mathbb Q)$ is mapped via $\gamma_*$ isomorphically
onto $\coniveau^nH^{2n+1}(X_{\cx},{\mathbb{Q}}(n))$.  
Such a correspondence $\pi$ defines a motive $(C,\pi)$, which determines an abelian
variety $J/K$\,;  more concretely, the correspondence $\pi$ induces  a $K$-morphism
$\Pi:J_{C/K}\to J_{C/K}$,
 and the image is the desired abelian variety $J/K$ with
 $H^1(J_{\overline K},\mathbb Q_\ell)$ identified with $\pi_*H^1(C_{\overline
K},\mathbb Q_\ell)$ as $\operatorname{Gal}(K)$-representations, by comparison isomorphisms.

The approach we take for finding this idempotent is motivated by the  following
well known,  
elementary lemma on  polarized Hodge structures\,; see \cite[Lem.~1.6]{Vial2} and \cite[Lem.~5]{voisinfilt}.   In the notation of
Proposition
\ref{P:niveau}, we will apply the lemma with $H=H^1(C_{\mathbb C},\mathbb Q)$, 
$H'=H^{2n+1}(X_{\mathbb C},\mathbb Q)$, and
$\gamma=\gamma_*$.

\begin{lem} \label{L:HS-Im} Let $(H,Q)$ and
  $(H',Q')$ be polarized rational Hodge structures of weight $i$ and     let
$\gamma : H
\rightarrow H'$ be a morphism of Hodge structures. Then $$H = \operatorname{Im}\, (\gamma') \oplus^{\perp_Q} \operatorname{Ker}\, (\gamma) \quad \text{and} \quad \operatorname{Im}\, (\gamma') = \operatorname{Im}\, (\gamma' \circ \gamma),$$ 
where $\gamma'$ is the transpose of $\gamma$ with respect to the polarizations\,; \emph{i.e.}, $\gamma' = Q \circ \gamma^\vee \circ Q'$ with $\gamma^\vee$ the dual of $\gamma$.\qed
\end{lem}

Going forward, we fix the standard principal polarization $Q$ on the rational
Hodge
structure $H^1(C_\mathbb{C},\mathbb{Q})$.   Recall that this is induced by the
cup-product, 
and that the identification
$H^1(C_\mathbb{C},\mathbb{Q}) = H^1(C_\mathbb{C},\mathbb{Q})^\vee \otimes
\mathbb{Q}(-1)$ via $Q$ is given by the action of the diagonal $\Delta_C \in
\operatorname{CH}^1(C\times C)_\mathbb{Q}$ on $H^1(C_\mathbb{C},\mathbb{Q})$.
Recall also that the dual morphism 
$$ 
(\gamma_*)^\vee :
H^{2n+1}(X_{\mathbb{C}},{\mathbb{Q}}(n))^\vee \rightarrow
H^{1}(C_{\mathbb{C}},{\mathbb{Q}})^\vee
$$
is identified via Poincar\'e duality  with the morphism of Hodge structures  
$$
 {}^t\gamma_*=\gamma^* : H^{2d-2n-1}(X_{\mathbb{C}},{\mathbb{Q}}(d-n))
\rightarrow
H^{1}(C_{\mathbb{C}},{\mathbb{Q}}(1)),
$$ 
where ${}^t\gamma$ is the transpose of $\gamma$ in
$\operatorname{CH}^{n+1}(X\times_KC)$.    We now show that  Theorem \ref{T:cycles'} will
follow if we can endow $H' = H^{2n+1}(X_{\mathbb{C}},{\mathbb{Q}}(n))$ with a polarization $Q'$ that is induced by  a correspondence defined over $K$.

 \begin{lem}\label{L:Mid}
Let $X$, $C$ and $\gamma$ be as in Proposition \ref{P:niveau}.
 Let $Q$ be the
standard polarization on $H^1(C_{\mathbb C},\mathbb Q)$.  If there are
 a
polarization  $Q'$ of  $H^{2n+1}(X_{\mathbb C},\mathbb Q(n))$ and a
correspondence
$\Theta\in  \operatorname{CH}^{2d-2n-1}(X\times_K X)_\mathbb{Q}$ making the
following diagram commute\,:
\begin{equation}\label{E:ThetaD}
\xymatrix@C=1em@R=2em{
H^{2n+1}(X_{\mathbb C})(n) \ar@{->}[r]^<>(.5){Q'} \ar@{=}[d] &
H^{2n+1}(X_{\mathbb C})(n)^\vee(-1) \ar[r]^<>(.5){(\gamma_*)^\vee}
\ar@{<->}[d]_{PD}& H^1(C_{\mathbb C})^\vee(-1) \ar[r]^<>(.5)Q \ar@{<->}[d]_{PD}&
H^1(C_{\mathbb C})\ar[r]^<>(.5){\gamma_*}& H^{2n+1}(X_\cx)(n)\\
H^{2n+1}(X_{\mathbb C})(n)\ar@{->}[r]^<>(.5){\Theta_*}& H^{2d-2n-1}(X_{\mathbb
C})(d-n) \ar[r]^<>(.5){{}^t\gamma_*}& H^1(C_{\mathbb C}) \ar@{=}[ru]& & \\
}
\end{equation}
then
 there exists an abelian subvariety $J\subseteq J_{C/K}$ defined over $K$,  and a correspondence defined over $K$, that induces for 
 all prime numbers  $\ell$
  a \emph{split} inclusion of $\operatorname{Gal}(K)$-representations  
\begin{equation}\label{E:cycles1}
H^1(J_{\bar K},\rat_\ell)\hookrightarrow H^{2n+1}(X_{\bar K},\rat_\ell(n)),
\end{equation}
with image $\coniveau^nH^{2n+1}(X_{\bar K},\rat_\ell(n))$. Moreover, there is a correspondence over $K$ inducing for all prime $\ell$ a splitting for \eqref{E:cycles1}.
\end{lem}

\begin{proof} As discussed above, up to the statement involving a splitting  for \eqref{E:cycles1}, it is enough to exhibit a correspondence $\pi \in \operatorname{CH}^1(C\times_K C)_{\rat}$ which acts idempotently on $H^1(C_{\cx},\mathbb Q)$ such that 
 $\pi_*H^1(C_{\cx},\mathbb Q)$ is mapped via $\gamma_*$ isomorphically
 onto $\operatorname{Im}\,(\gamma_*) = \coniveau^nH^{2n+1}(X_{\cx},{\mathbb{Q}}(n))$.  
Following the notation of Lemma \ref{L:HS-Im}, set $\gamma_*'=Q\circ
(\gamma_*)^\vee \circ Q'$, and set $p:H^1(C_{\mathbb C},\mathbb Q)\to
\operatorname{Im}\,(\gamma_*')$ to be the orthogonal projection.  
Consider the endomorphism
$$
\gamma_*'\circ \gamma_* = {}^t\gamma_*\circ \Theta_* \circ
\gamma_* : H^1(C_\mathbb{C},\mathbb{Q}) \rightarrow
H^1(C_\mathbb{C},\mathbb{Q}).
$$
From Lemma \ref{L:HS-Im}, the endomorphism  $\gamma_*'\circ \gamma_*$ induces an
automorphism of $\operatorname{Im}(\gamma_*')$.  
 By the Cayley--Hamilton theorem, there is a polynomial $P$ with
$\mathbb{Q}$-coefficients such that the inverse of
$(\gamma_*'\circ \gamma_*)|_{\operatorname{Im}(\gamma_*')}$ is
$P((\gamma_*'\circ \gamma_*)|_{\operatorname{Im}(\gamma_*')})$. It is then
apparent that $p$ is
induced by 
\begin{equation} \label{E:pi}
\pi := {}^t\gamma \circ \Theta \circ \gamma  \circ P({}^t\gamma
\circ \Theta \circ \gamma) \in \operatorname{CH}^1(C\times_K C)_\rat,
\end{equation}
since by construction it acts as the identity on $\operatorname{Im}(\gamma_*')$
and as zero on $\operatorname{Ker}(\gamma_*)$. That $\operatorname{Im}\,(\gamma \circ \pi)_*= \operatorname{Im}\,(\gamma_*)$ follows from Lemma \ref{L:HS-Im}.

Finally, concerning the splitting for \eqref{E:cycles1}, writing the correspondence $\pi$ of formula \eqref{E:pi} as $\pi =
\varphi \circ \gamma$, we see that a $\operatorname{Gal}(K)$-equivariant
splitting to the inclusion $$\gamma_* : H^1(J_{\overline K},{\mathbb{Q}}_\ell)
\cong \pi_*H^1(C_{\bar K},\rat_\ell)\hookrightarrow H^{2n+1}(X_{\bar
K},\rat_\ell(n))$$ is given by the action $\varphi_*$ of $\varphi$.
\end{proof}

Up to proving Lemma \ref{L:Q'Th}, the proof of Theorem \ref{T:cycles'} is now complete.
\end{proof}

\begin{rem}\label{R:Eq-motives} 
In motivic terms, we have proved the following. With the notations of \cite{Andre}, denote $h^1(J)$ the weight-$1$ direct summand of the homological motive of $J$, and  $h(X)$ the homological motive of $X$. We proved that the
 morphism of motives $$h^1(J) \cong (C,\pi) \stackrel{\gamma\circ \pi}{\longrightarrow}
 h(X)(n)$$ is split injective with a splitting given by the correspondence $\pi
 \circ \varphi : h(X)(n) \to (C,\pi)$, and that the cohomology of
$\operatorname{Im}\,(\gamma \circ \pi)$ is $\coniveau^n H^{2n+1}(X_{\bar
  K},\rat_\ell(n))$.
\end{rem}

\begin{lem} \label{L:Q'Th}
Let $X$ be a smooth projective variety of pure dimension
$d$ over $K\subseteq \cx$.  Assume any of (a)--(d) of Theorem \ref{T:cycles'}.
 Then there is a correspondence $\Theta \in \operatorname{CH}^{d-2n-1}(X\times_K
X)$ such that the pairing $(x,y)\mapsto \int_{X_\cx} x \cup \Theta_* y$ defines
a polarization $Q'$ on $H^{2n+1}(X_\cx,\rat(n))$\,; \emph{i.e}., such that
\eqref{E:ThetaD} is commutative.
\end{lem}

\begin{proof}
Let us first recall how to polarize the Hodge structure $H^n(X_\cx,\rat)$ using
the Lefschetz decomposition formula, for any smooth projective variety $X$ of
pure dimension $d$.
 We
fix the class $L$ of a smooth hyperplane section of $X$. By the Lefschetz
hyperplane theorem, the powers of $L$ induce isomorphisms $L^{d-n} :
H^n(X_\cx,\rat) \rightarrow H^{2d-n}(X_\cx,\rat(d-n))$ for $n \leq d$.
By Poincar\'e duality, the bilinear form $Q_n$ on $H^n(X_\cx,\rat)$ given by
$$Q_n(x,y) = \int_{X_\cx} x\cup L^{d-n} y$$ is non-degenerate.
By definition, for $n\leq d$, the primitive part of $H^n(X_\cx,\rat)$ is
$$H^n(X_\cx,\rat)_{\mathrm{prim}} := \operatorname{Ker} \big(L^{d-n+1} :
H^n(X_\cx,\rat) \rightarrow
H^{2d-n+2}(X_\cx,\rat)\big).$$ We have the Lefschetz decomposition formula
\begin{equation}
\label{E:Ldecomp}
H^n(X_\cx,\rat) = \bigoplus_{0 \leq r} L^r
H^{n-2r}(X_\cx,\rat(-r))_{\mathrm{prim}}.
\end{equation} The Hodge index theorem states that
decomposition \eqref{E:Ldecomp} is orthogonal for $Q_n$ and that the sub-Hodge
structure $L^r
H^{n-2r}(X_\cx,\rat(-r))_{\mathrm{prim}}$ endowed with the form
$(-1)^rQ_n$ is polarized. We let $p^{n,r}$ denote the orthogonal
projector
$$H^n(X_\cx,\rat) \rightarrow L^r
H^{n-2r}(X_\cx,\rat(-r))_{\mathrm{prim}}  \rightarrow H^n(X_\cx,\rat),$$
and we define $$s_n := \sum_r (-1)^r p^{n,r} \quad \mbox{for } n\leq d, \quad
\mbox{and} \quad s_n := \sum_r (-1)^r\, {}^tp^{2d-n,r} \quad \mbox{for } n> d.$$
With these notations, the bilinear form $Q'$ on $H^n(X_\cx,\rat)$ given by
$Q'(x,y) = \int_{X_\cx} x\cup L^{d-n}s_n y$ is a polarization. (Here, $L^r$ for
$r<0$ is defined to be the inverse of $L^{-r}$.)

Let us now consider the polarized Hodge structure
$H^{2n+1}(X_\cx,\rat(n))$ and
proceed to the proof of Lemma~\ref{L:Q'Th}.

(a) If $2n+1=1$, then note that $H^{1}(X_\cx,\rat) =
H^{1}(X_\cx,\rat)_{\mathrm{prim}}$ so that $s_1$ is induced by the diagonal
$\Delta_X \in \operatorname{CH}^d(X\times X)_{\rat}$. Therefore $\Theta =
L^{n-1}$ is suitable. (In fact, in the case $n=0$, Theorem \ref{T:cycles'} holds
with $J=\operatorname{Pic}^0_X$.) 

If $2n+1=3$, let us write $\Lambda \in \operatorname{CH}^1(X\times X)_\rat$ for
a
correspondence that induces the inverse to the Lefschetz isomorphism $L^{d-1} :
H^{1}(X_\cx,\rat) \stackrel{\simeq}{\rightarrow} H^{2d-1}(X_\cx,\rat(d-1))$
(Such a correspondence $\Lambda$ does exist\,; see for instance the proof of
Proposition
\ref{P:niveau}.)
We define $$P^{3,1} := L\circ \Lambda \circ L^{d-2} \in
\operatorname{CH}^d(X\times X)_\rat.$$ It is clear that $P^{3,1}$ acts as
$p^{3,1}$ on $H^3(X_\cx,\rat(1))$\,; we may then define $P^{3,0} = \Delta_X -
P^{3,1}$, and $S_3 := P^{3,0} - P^{3,1} \in \operatorname{CH}^d(X\times X)_\rat$
acts on $H^3(X_\cx,\rat(1))$ as $s_3$. Hence $\Theta := L^{d-2} \circ S_3$ is
suitable.

(b) If $X$ satisfies the Lefschetz standard conjecture, then Kleiman
\cite{kleiman} has shown that the projectors $p^{2n+1,r}$ are induced by
correspondences $P^{2n+1,r} \in \operatorname{CH}^d(X\times X)_\rat$. Therefore,
denoting $\Lambda \in \operatorname{CH}^{2d-2n-1}(X\times X)_\rat$ either
the correspondence $L^{d-2n-1}$ if $2n+1 \leq d$ or a correspondence inducing
the inverse of the
action of $L^{2n+1-d}$ on $H^{2d-2n-1}(X_\cx,\rat(d-n-1))$ if $2n+1 > d$, we get
that  $\Theta := \Lambda \circ \big( \sum_r (-1)^rP^{2n+1,r}\big)$ is
suitable.

(c) Let $Y$ be a smooth projective threefold, 
 and let
$\Gamma \in
\operatorname{CH}^{n+2}(Y\times_K
X)_\rat$ be such that $$\Gamma_* : H^3(Y_\cx,\rat(1)) \rightarrow
H^{2n+1}(X_\cx,\rat(n))$$ is surjective. The polarization on
$H^3(Y_\cx,\rat(1))$ coming from the Lefschetz decomposition
induces a polarization on $H^{2n+1}(X_\cx,\rat(n))$ and we are going to show
that this can be done algebro-geometrically. By (a), there is a correspondence
$S_Y
\in \operatorname{CH}^3(Y\times Y)_\rat$ such that the pairing $(x,y ) \mapsto
\int_{Y_\cx} x \cup (S_Y)_*y$ defines a polarization on $H^{3}(Y_\cx,\rat(1))$.
By Lemma \ref{L:HS-Im}, we get a decomposition 
$$H^{3}(Y_\cx,\rat(1)) = \operatorname{Im}((S_Y)_* \circ \Gamma^* \circ
L^{d-2n-1}) \oplus \operatorname{Ker}(\Gamma_*).$$
 It follows that $\Gamma \circ S_Y \circ {}^t\Gamma \circ L^{d-2n-1}$ induces an
automorphism of $H^{2n+1}(X_\cx,\rat(n))$. 
Consider the correspondence
 $$\Gamma' := S_Y \circ {}^t\Gamma \circ L^{d-2n-1} \in
\operatorname{CH}_{n+2}(X\times
Y)_\rat.$$
The bilinear form $$Q'(x,y) := \int_{Y_\cx} (\Gamma')_*x \cup
(S_Y)_*(\Gamma')_*y = \int_{X_\cx} x \cup (\Gamma')^*(S_Y)_*(\Gamma')_*y$$ then
defines a polarization on
$H^{2n+1}(X_\cx,\rat(n))$. The correspondence $\Theta := {}^t\Gamma' \circ S_Y
\circ \Gamma'$ is then suitable.

(d)  is a special case of (c) because if $H^{2n+1}(X_\cx,\rat(n))$ is of
geometric coniveau $n$, then by Proposition \ref{P:niveau},
$H^{2n+1}(X_\cx,\rat(n))$ is of geometric niveau $n$ (meaning that there is a
 smooth projective curve $C$ over $K$ and a correspondence $\gamma \in
\operatorname{CH}^{n+1}(C\times_K X)$ such that $\gamma_* : H^1(C_\cx,\rat) \to
H^{2n+1}(X_\cx,\rat(n))$ is surjective) and hence of geometric niveau $n-1$.
\end{proof}

\section{Algebraic representatives over algebraically closed fields} \label{S:c2c}

In this section we review Murre's work \cite{murre83} on algebraic representatives over algebraically closed fields. The theory of algebraic representatives is indeed the starting point to showing in Theorem \ref{T:cycles3'} that the image of the Abel--Jacobi map  $AJ:\operatorname{A}^{2}(X_{\mathbb C})\to
 J^{3}(X_{\mathbb C})$
descends to a field of definition of $X$.  
In \S \ref{S:prelbasechange}, we recall the functorial properties  of Chow and Albanese schemes under  field extension, and
similar results for homomorphisms of abelian varieties. This is used in \S
\ref{S:bc1} to  establish algebraically closed base change for algebraic
representatives\,; see our main Theorem \ref{T:ACBC}.  \linebreak

\subsection{Preliminaries on Chow groups and algebraic
 representatives}\label{S:ChGAb}

Fix an algebraically closed field $k$.  
Recall following Samuel (see \cite[Def.~1.6.1]{murre83} or
\cite[2.5]{samuelICM})
 that given an abelian variety $A$ over $k$, a
homomorphism of groups
$$
\begin{CD}
\operatorname{A}^i(X)@>\phi>> A(k)
\end{CD}
$$
is said to be \emph{regular} if for every pair $(T,Z)$ with $T$ a  pointed
smooth integral  
variety,
and
$Z\in \operatorname{CH}^i(T\times X)$,  
the composition 
$$
\begin{CD}
T(k)@> w_Z >> \operatorname{A}^i(X)@>\phi >>A(k)
\end{CD}
$$ 
is induced by a morphism of varieties $\psi_Z:T\to A$ over $k$, where, if $t_0\in T(k)$
is
the base point of $T$,  $w_Z:T(k)\to \operatorname{A}^i(X)$ is   given by
$t\mapsto Z_t-Z_{t_0}$\,; here $Z_t$ is the refined Gysin fiber (see e.g.,
\cite[Ch.10]{fulton}).

Given a \emph{surjective} regular homomorphism $\phi :
\operatorname{A}^i(X) \rightarrow A(k)$, there exists a correspondence
 $Z \in
\operatorname{CH}^i(A\times X)$ such that $\phi \circ w_Z$ is an isogeny (e.g.,
\cite[Cor.~1.6.3]{murre83}\,; see also Lemma \ref{L:M-1.6.2}).

\begin{dfn}\label{D:AlgRep}
 An \emph{algebraic representative} for $\operatorname{A}^i(X)$ is an abelian
 variety that is initial for regular homomorphisms.  In other words, it is a pair
 $(\operatorname{Ab}^i(X),\phi_X^i)$  with $\operatorname{Ab}^i(X)$ an abelian
 variety over $k$, and 
 $$
 \xymatrix{
  \operatorname{A}^i(X)\ar@{->}[rr]^{\phi_X^i} && \operatorname{Ab}^i(X)(k)\\
 }
 $$
 a regular homomorphism of groups such that for every  pair $(A,\phi)$
 consisting of an abelian variety $A$ over $k$ and a regular homomorphism of
 groups 
 $\phi:\operatorname{A}^i(X)\to A(k)$, there exists  a unique morphism
 $f:\operatorname{Ab}^i(X)\to A$ of varieties over $k$ such that the following
 diagram commutes
 $$
 \xymatrix@C=.3cm@R=.5cm{
  \operatorname{A}^i(X)\ar@{->}[rr]^<>(0.5){\phi_X^i}
  \ar@{->}[rd]_{\phi}&&\ar@{-->}[ld]^{f(k)}_<>(0.6){\exists
   !}\operatorname{Ab}^i(X)(k).\\
  & A(k)& \\
 }
 $$
\end{dfn} 
Algebraic representatives, if they exist, are unique up to unique isomorphism.  
For an algebraic representative, the homomorphism $\phi_X^i$ is surjective
(e.g., \cite[Rem.~a), p.225]{murre83}).
It is well known that  algebraic representatives of $\A^1(X)$ and $\A^d(X)$ (where $d$ is the dimension of $X$) do exist\,:  the Abel--Jacobi map taking a divisor
to its associated line bundle
$A^1(X) \to \operatorname{Pic}^0(X)(k)$ and  the Albanese map $A^d(X) \to
\operatorname{Alb}(X)(k)$ are universal surjective regular homomorphisms.

We will frequently use the following necessary and sufficient condition for the existence of an
algebraic representative.

\begin{pro}[{\cite[Thm.~2.2]{hsaito}, \cite[Prop.~2.1]{murre83}}] \label{P:saito}
 An algebraic representative for $\operatorname{A}^i(X)$ exists if and only if there exists a constant $M$ such that for every surjective regular homomorphism $\phi : \operatorname{A}^i(X) \twoheadrightarrow A(k)$ we have $\dim_k A \leq M$.
\end{pro}

\begin{rem}
 Note that the notions of regular homomorphism and of algebraic representative
 are usually defined for smooth projective \emph{irreducible} varieties over an
 algebraically closed field $k$. However it is harmless to
 consider 
 possibly disconnected smooth projective varieties\,: for instance, 
 a smooth projective variety $X$ over  $k$ admits an algebraic representative $(\operatorname{Ab}^i(X),\phi^i_X)$
 for $\operatorname{A}^i(X)$ if and only if each of its irreducible components $X_1,\ldots,X_n$ admits an algebraic
 representative $(\operatorname{Ab}^i(X_j),\phi^i_{X_j})$ for
 $\operatorname{A}^i(X_j)$, $j=1,\ldots,n$. In this case, 
 $(\operatorname{Ab}^i(X),\phi^i_X)\cong
 \bigoplus_{j=1}^n(\operatorname{Ab}^i(X_j),\phi^i_{X_j})$.
\end{rem}

\subsection{Results of Murre on codimension two  cycles}\label{S:M-BS}

We recall  Murre's main result concerning algebraic representatives for codimension two cycles.

\begin{teo}[{\cite[Thm.~A, p.226, Thm.~C,
  p.229, Thm.~10.3, p.258]{murre83}}] \label{T:Murre} Let $X$ be a smooth
 projective 
 variety over an algebraically closed field $k$. 
 \begin{alphabetize}
  \item  There exists an
  algebraic representative $(\operatorname{Ab}^2(X),\phi^2_X)$ of
  $\operatorname A^2(X)$.  
  \item For all but finitely many primes
  $\ell \ne \operatorname{char} k$, there is
  a natural inclusion
  \begin{equation}\label{E:MT1}
  H^1(\widehat{\Ab^2(X)},\integ_\ell) \hookrightarrow  H^3(X,\integ_\ell(1)).
  \end{equation}
  
  \item For all primes $\ell \ne \operatorname{char} k$, there is a
  natural inclusion 
  \begin{equation}\label{E:MT1.5}
  H^1(\widehat{\Ab^2(X)},\rat_\ell) \hookrightarrow H^3(X,\rat_\ell(1)).
  \end{equation}

  \item If $k=\mathbb C$, then an algebraic representative is given by
  the image of the Abel--Jacobi map.  In other words, the pair
  $(J_{\operatorname{a}}^3(X),\psi^2)$, where $$\psi^2:\operatorname
  A^2(X)\to J^3(X)$$ is the Abel--Jacobi map, and
  $J_{\operatorname{a}}^3(X)=\operatorname{Im}(\psi^2)$, is an
  algebraic representative of $\operatorname A^2(X)$.  For {\em each }
  $\ell$, the Abel--Jacobi map $\psi^2$ is an injection on torsion, so that 
  there is an isomorphism\,:
  \begin{equation}\label{E:MT2}
  \xymatrix@C=3em{
   T_\ell \operatorname{A}^2(X) \ar@{->}[r]^{T_\ell \psi^2}_\cong 
   & T_\ell J_{\operatorname{a}}^3(X).\\
  }
  \end{equation}
 \end{alphabetize}
\end{teo}

\begin{rem}\label{R:MuMaps}
 The inclusions \eqref{E:MT1} and \eqref{E:MT1.5} are not stated explicitly in
 \cite{murre83}\,; the argument is contained in Murre's argument that
 the second Bloch map is injective \cite[Prop.~9.2, p.256]{murre83}.
 More precisely, \cite[(18), p.250]{murre83} establishes that for
 almost all $\ell$ there is an inclusion
 \begin{equation}\label{E:degWZ}
 H^1(\widehat {\Ab^2(X)},\mathbb Z_\ell(1))=H^1(\Ab^2(X),\mathbb
 Z_\ell)^\vee=T_\ell\Ab^2(X)\subseteq
 T_\ell\operatorname{A}^2(X)\subseteq T_\ell\operatorname{CH}^2(X).
 \end{equation}
 (For the finitely many exceptional $\ell \not = \operatorname{char}
 k$, Murre's argument shows that there is still an inclusion of {\em
  rational} Tate modules $H^1(\widehat{\Ab^2(X)},\rat_\ell(1)) = V_\ell
 \Ab^2(X) \subseteq V_\ell  \chow^2(X)$.)
 Then, Murre's argument that the second Bloch map is injective shows that for
 almost all $\ell$ there is also an inclusion\,:
 \begin{equation}\label{E:ChtoH3}
 T_\ell \operatorname{CH}^2(X)\hookrightarrow H^{3}(X,\mathbb
 Z_\ell(2)).
 \end{equation}
 (This may fail to be an inclusion at the primes for which
 $H^{3}(X,\mathbb Z_\ell(2))$  has torsion\,; for cohomology
 with $\mathbb Q_\ell$-coefficients, this does not pose a problem, and one
 obtains the inclusion for all $\ell$.)
 The inclusion \eqref{E:ChtoH3} is not explicitly stated in \cite{murre83} but follows immediately from  \cite[(19)]{murre83} and \cite[Lem.~2.4]{bloch79}, together with the fact that  $H^{3}(X,\mathbb Z_\ell(2))$
 is torsion-free for almost all $\ell$ by a theorem of Gabber
 \cite{gabbertorsion}.
\end{rem}

\begin{rem}
 We can be more precise about the primes $\ell$ for which the argument above does
 not
 establish the inclusion \eqref{E:MT1}.   By \cite[Lem.~1.6.2, p.224]{murre83}
 there exists a cycle 
 $Z$ on $X\cross \Ab^2(X)$
 such that the composition $\psi_Z:=\phi_X\circ w_Z: \Ab^2(X) \ra \A^2(X) \ra
 \Ab^2(X)$ is  an isogeny. 
 For each $\ell \nmid \deg(\psi_Z)$, Murre   \cite[p.250]{murre83} establishes  
 \eqref{E:degWZ}.  
 Therefore, \eqref{E:MT1} holds for all primes $\ell$ such that $\ell \nmid
 \deg(\psi_Z)$ and such that $H^{3}(X,\mathbb Z_\ell(2))$ is
 torsion-free.   Over $\mathbb C$, a result of Voisin \cite[Thm.~1.12,
 1.13]{voisinAJ13}, \cite[Thm.~1.7]{voisinUniv} establishes 
 that for threefolds, 
 if the diagonal admits an \emph{integral}
 cohomological decomposition,  
 then there 
 exists a cycle $Z$ so that $\psi_Z$ is of degree $1$ 
(see also \cite[Rem.~1.9]{voisinAJ13}). 
\end{rem}

\subsection{Algebraic representatives, Chow varieties, and traces for primary
 field extensions.}\label{S:prelbasechange}

If an algebraic representative exists, then by definition it is dominated by the
Albanese varieties of components of the Chow scheme.  Since 
Albanese varieties and Chow schemes respect base change of fields, this  gives an approach for studying
the problem of base change for algebraic representatives.   A key tool for
studying base change of abelian varieties for  primary extensions of fields is
the so-called $L/K$-trace.  We review in this section some of the results in the
literature, and how we will employ them in the context of algebraic
representatives.

\subsubsection{Chow rigidity and $L/K$-trace for primary field extensions}
\label{S:rigid}
We exploit
(Chow) rigidity and $L/K$-traces for  abelian
varieties and primary field extensions.  See
\cite{conradtrace} for a modern treatment\,; here, we summarize 
what we need.  Let $L/K$ be a primary extension of fields, i.e., an
extension such that the algebraic closure of $K$ in $L$ is purely
inseparable over $K$.  If $L/K$ is primary and separable (e.g., if $K$ is algebraically closed), then $L/K$ is said to be regular.

If $A$ is an abelian variety over $L$,  an \emph{$L/K$-trace} is a final object
$(\operatorname{tr}_{L/K}(A),\tau_A)$ in the category of pairs $(\underline
B,f)$ where $\underline B$ is an abelian variety over $K$ and $f:\underline
B_L\to A$ is a map of abelian varieties over $L$.
An $L/K$-trace exists \cite[Thm.~6.2]{conradtrace}, and is unique up to unique
isomorphism.     For brevity, we will frequently denote the $L/K$-trace by the pair $(\LKtrace A, \tau) $.

The kernel $\ker(\tau)$ is always a finite group scheme \cite[Thm.~6.4(4)]{conradtrace}.  If $\operatorname{char}(K) = 0$, then $\ker(\tau)$ is trivial \cite[p.72]{conradtrace}\,; more generally, if $L/K$ is regular, then $\ker(\tau)$ is a connected group scheme with connected dual
\cite[Thm.\ 6.4, Thm.\
6.12]{conradtrace}, and thus $\tau$ is purely inseparable.
In particular \cite[p.76]{conradtrace}, if $L/K$ is regular, then on $L$-points $\tau$ induces an
injection
\begin{equation}
\label{eq:rigidpoints}
\xymatrix{
 \LKtrace A(L) \ar@{^(->}[r]^{\tau(L)}& A(L).
}
\end{equation}

The image of
$\tau$ in $A$ descends to $K$ if and only if $\operatorname{Ker}(\tau)$ is the
trivial group
scheme \cite[p.72]{conradtrace}. As we have seen,  this is automatic in characteristic zero, but need not hold in positive characteristic.

Chow's rigidity theorem  (see \cite[Thm.~3.19]{conradtrace}) states that if 
$\underline B$ and $\underline C$ are both abelian varieties over $K$ (and $L/K$
is still
primary), then  the natural inclusion 
$$
\Hom_K(\underline B,\underline C) \to \Hom_L(\underline B_L,\underline C_L)
$$ 
is
a bijection.  In particular, if $A/L$ and $\underline B/K$ are abelian
varieties,
then  there are canonical bijections
\begin{equation}
\label{eq:rigidtrace}
\Hom_L(\underline B_L,A) =\operatorname{Hom}_L(\underline B_L, \LKtrace A_L)=
\Hom_K(\underline B, \LKtrace A),
\end{equation}
where the first equality is  the universal property of the $L/K$-trace,
and the second equality comes  from rigidity.

\subsubsection{Chow varieties, Albaneses, and algebraic representatives}
\label{S:chow}
If $X/K$ is a  projective variety over a perfect field
$K$, then
for each $i$ there is a Chow scheme $\Chow^i(X)$ over $K$, a union of countably
many connected components, whose points parametrize cycles of codimension $i$.
If $L/K$ is an extension of fields and $X_L:=X_K\times _KL$, then $\Chow^i(X_L)
\iso
\Chow^i(X)_L:=\operatorname{Chow}^i(X)\times_KL$
\cite[Prop.~1.1]{friedlander91}\,; see also
\cite{K96}.  

If $(Y,y)/K$ is a pointed geometrically integral
variety 
over a perfect field,
an
Albanese variety for $Y$ is  an abelian variety $\alb(Y)$ over $K$ equipped
with a pointed map $(Y,y) \ra (\alb(Y),0)$ that is initial  with
respect to pointed maps from $(Y,y)$ to abelian varieties.  If  $Y/K$ is
smooth
and proper,
then $\alb(Y) = \operatorname{Pic}^0_{(\pic^0_{Y/K})_{{\rm red}}/K}$\,; since the
Picard functor
behaves well under base change, we have $\alb(Y_L) \iso
\alb(Y)\cross_K L$ for any extension of perfect fields $L/K$.  For 
geometrically integral  
varieties  $Y/K$  that are not smooth and proper,  either using (a
compactification and) resolution of
singularities in characteristic $0$, or
an  argument via de Jong's alterations as in \cite[Cor.~A.11]{mochizuki12}, one
may again show that $Y$ admits an Albanese, and that 
$\alb(Y_L) \iso\alb(Y)\cross_K L$.

One setting where we will use Albaneses extensively (e.g., the proof of Theorem \ref{T:ACBC}) is the following.  
Let $X$ be a smooth projective variety defined over an algebraically
closed
field  $k$, and assume
there is an  algebraic representative
$(\operatorname{Ab}^i(X),\phi^i_X)$.  From the discussion above, we see that $\operatorname{Ab}^i(X)$ 
is the colimit of  the inverse system of abelian
varieties obtained from the Albanese varieties of components of
the Chow scheme
$\operatorname{Chow}^i(X)$.  
In fact, one should work with Albaneses of smooth alterations of the reduction of  products of components of the Hilbert scheme, but we avoid this technical point in this discussion.  
 In particular, $\operatorname{Ab}^i(X)$ admits a surjection
from a finite product of these 
 abelian varieties.  
Moreover, if $K\subseteq k$ is a perfect field, and
$\underline X$ is defined over $K$ with $X=\underline X\times_K k$, then 
$\Ab^i(X)$ admits a morphism from  the pullback to $k$ of the inverse system of torsors under abelian
varieties defined over $K$ obtained from the Albanese torsors of components of
the Chow scheme
$\operatorname{Chow}^i(\underline X)$. 
In particular, since every component of $\operatorname{Chow}^i(X)$ is a component of a pullback to $k$ of a component of  $\operatorname{Chow}^i(\underline X)$, we have that $\Ab^i(X)$ admits a surjection from a torsor under an abelian variety defined over $K$.  
(In fact, it will follow from Theorems  \ref{T:ACBC} and \ref{T:Desc} that 
if $k=\bar  K$, or $\operatorname{char}(k)=0$,  then  $\Ab^i(X)$, if it exists, descends to an abelian variety  over $K$.)

\subsection{Algebraically closed base change for algebraic
 representatives}
\label{S:bc1}
The aim of this subsection is to establish our main addition to Murre's work on algebraic representatives over algebraically closed fields\,:

\begin{teo}\label{T:ACBC}
 Let $\Omega/k$ be an extension of algebraically closed fields.  Let
 $X/k$ be a smooth projective variety.  Assume that either $\A^i(X)$  or 
 $\A^i(X_\Omega)$  admits an
 algebraic
 representative. Then both   $\A^i(X)$ and
 $\A^i(X_\Omega)$  admit an
 algebraic
 representative.
 
 Moreover, 
 let 
 $(\operatorname{Ab}^i(X),\phi^i_{X})$ and  $(\operatorname{Ab}^i(X_{\Omega}),\phi^i_{X_\Omega})$ be algebraic representatives for   $\A^i(X)$ and
 $\A^i(X_\Omega)$, respectively.
 Then there are natural maps
 \[
 \xymatrix{
  \Ab^i(X) \ar[r]^{b'} & \LKtrace \Ab^i(X_\Omega)
 }\text{ and } \xymatrix{
 \LKtrace \Ab^i(X_\Omega)_\Omega \ar[r]^{\tau} & \Ab^i(X_\Omega)
}
\]
over $k$ and $\Omega$, respectively,  where 
$(\LKtrace{\Ab}^i(X_\Omega),\tau)$ is the $\Omega/k$-trace of
$\operatorname{Ab}^i(X_\Omega)$ (\S \ref{S:rigid}).
In characteristic zero, $b'$ and $\tau$ are isomorphisms\,; in positive 
characteristic, they are purely inseparable isogenies.
In particular, in characteristic zero, there is a natural isomorphism $\operatorname{Ab}^i(X)_\Omega \cong \operatorname{Ab}^i(X_{\Omega})$.
\end{teo}

We start by recalling a rigidity result for torsion in  Chow
groups.

\begin{teo}[{Lecomte \cite[Thm.~3.11]{lecomte86}}]
 \label{P:acbc}
 Let $\Omega/k$ be an extension of algebraically closed fields, and let $X$ be a
 projective variety over $k$.   Base change on cycles induces an
 injective
 homomorphism of Chow groups $b_{\Omega/k}: \chow^i(X) \ra \chow^i(X_\Omega)$. For each positive integer $N$, $b_{\Omega/k}$ induces isomorphisms on the
 $N$-torsion groups\,:
 \begin{alphabetize}
  \item \label{P:acbc1}  $b_{\Omega/k}[N]:\chow^i(X)[N]
  \stackrel{\sim}{\longrightarrow } \chow^i(X_\Omega)[N]$\,; 
  
  \item  \label{P:acbc2}  $b_{\Omega/k}[N]:\A^i(X)[N]
  \stackrel{\sim}{\longrightarrow } \A^i(X_\Omega)[N]$. 
 \end{alphabetize}
\end{teo}

\begin{proof} That the base-change homomorphism $b_{\Omega/k}: \chow^i(X) \ra
 \chow^i(X_\Omega)$  is injective is standard and is proved using a spreading out
 argument followed by a specialization argument (e.g.,
 the argument of \cite[Lem.~1A.3, p.22]{blochbook}). Precisely, denoting $r_{\Omega/k} : \chow^i(X_\Omega) \to \chow^i(X)$ the specialization homomorphism (see \cite[Lem.~1A.3, p.22]{blochbook} and \cite[\S 20.3]{fulton}), we have that $r_{\Omega/k} \circ b_{\Omega/k}$ is the identity on $\chow^i(X)$, giving the injectivity of $b_{\Omega/k}$.  
 Part  \ref{P:acbc1} is the content of \cite[Thm.\ 3.11]{lecomte86}. 
As a consequence of \ref{P:acbc1}, $r_{\Omega/k}$ and $b_{\Omega/k}$ are in fact inverses of each other on $N$-torsion.  
Part \ref{P:acbc2} then follows from the fact that $b_{\Omega/k}$ and $r_{\Omega/k}$ respect algebraic equivalence.
\end{proof}

\begin{proof}[Proof of Theorem \ref{T:ACBC}] We proceed in four steps.

 \noindent  \textbf{Step 1.}  \emph{Let $X$ be a smooth projective variety defined over $k$, let  $A/k$ be an abelian variety, and let $\phi : \operatorname{A}^i(X) \rightarrow A(k)$ be a regular homomorphism.  Then there is a regular homomorphism $\phi_\Omega: \operatorname{A}^i(X_\Omega) \to
  A_\Omega(\Omega)$,  which is surjective if $\phi$ is surjective,
  making the following diagram commute\,:}
 \begin{equation}\label{E:ACBC3}
 \xymatrix{
  \operatorname{A}^i(X)\ar@{->}[r]^<>(0.5){  \phi} \ar@{->}[d]^{ b} &
  A(k)  \ar@{->}[d]^{\text{base change}}  \\
  \operatorname{A}^i(X_\Omega)  \ar@{->}[r]^<>(0.5){\phi_\Omega}&
  A_\Omega(\Omega).
  \\
 }
 \end{equation}
 Note that  by the criterion of \cite[Thm.~2.2]{hsaito}, \cite[Prop.~2.1]{murre83} (see Proposition \ref{P:saito}), it follows from Step 1 that if $\operatorname{A}^i(X_\Omega)$ has an algebraic representative, then so does $\operatorname{A}^i(X)$.

 To establish Step 1, we first 
 construct $\phi_\Omega$\,; we thank the referee for providing such a construction.
 Denote
 by $\operatorname{A}^i_\natural (X)$ the contravariant functor on the category
 of 
 smooth integral  varieties over $k$ given by families of
 algebraically
 trivial cycles on $X$. More precisely, the functor $\operatorname{A}^i_\natural
 (X)$ sends a  variety $T$ as above to the group of cycles $Z \in 
 \operatorname{CH}^i(X \times_k T)$ such that for some (equivalently, for any)
 geometric point $x: \operatorname{Spec}(\Omega') \to  T$ with $\Omega'$ being an
 algebraically closed field over $k$, we have that the cycle $x^! (Z) \in 
 \operatorname{CH}^i(X_{\Omega'})$
 is algebraically trivial.
 By definition of a regular homomorphism, we have a morphism of contravariant
 functors
 $$
 \phi_\natural :\operatorname{A}^i_\natural(X)\to A
 $$
 on the category of smooth  integral  varieties over $k$.  
 Moreover, both functors $\operatorname{A}^i_\natural(X)$ and $A$  send inverse limits of varieties to direct limits of
 abelian groups. 
 Therefore, the  morphism of functors $\phi_\natural$  extends in a
 canonical way to the category of schemes over $k$ that can be obtained as
 inverse limits of smooth integral varieties over $k$.
 In particular,
 evaluating on $\operatorname{Spec}\Omega$, 
 we obtain the desired  regular
 homomorphism 
 $$
 \begin{CD}
 \operatorname{A}^i(X_\Omega)=\operatorname{A}^i_\natural(X)(\Omega)
 @>\phi_\natural(\Omega)>>  A(\Omega),
 \end{CD}
 $$ 
 which we denote by $\phi_\Omega$. Note that this homomorphism is surjective when $\phi$ is surjective due
 to  the existence of the cycle and isogeny from \cite[Cor.~1.6.3]{murre83} (see Lemma \ref{L:M-1.6.2}(c)).

 \noindent  \textbf{Step 2.} \emph{Let $X$ be a smooth projective variety defined over $k$, let $A/\Omega$ be an abelian variety, and let $\phi_{X_\Omega} : \operatorname{A}^i(X_\Omega) \to A(\Omega)$ be a regular homomorphism. 
  Then  there is a regular homomorphism $\LKtrace{\phi}:
  \operatorname{A}^i(X)\to \LKtrace{A}(k)$ to the $k$-points of the $\Omega/k$-trace of
  $A$ (see \S \ref{S:rigid}), making the
  following diagram commute\,:}
 \begin{equation}\label{E:LKphi}
 \xymatrix@R=.5cm{
  \operatorname{A}^i(X) \ar[r]^<>(0.5){\LKtrace{\phi}} \ar[dd]^b&
  \LKtrace{A}(k) \ar[d]^{\text{ base change}}\\
  & \LKtrace{A}(\Omega) \ar[d]^{\tau(\Omega)}\\
  \operatorname{A}^i(X_\Omega)  \ar[r]^<>(0.5){\phi_{X_\Omega}}&
  A(\Omega).
 }
 \end{equation} 
 Here,
 $b=b_{\Omega/k}:\operatorname{A}^i(X)\to
 \operatorname{A}^i(X_\Omega)$ is the base change map on cycles, and
 $\tau:\LKtrace{A}_\Omega \to A$ is the
 canonical map (see \S \ref{S:rigid})\,; recall that $\tau$ is an isomorphism onto its image in characteristic zero, and is at worst a purely inseparable isogeny onto its image in positive characteristic.  In particular, $\tau$ induces an inclusion on $\Omega$-points \eqref{eq:rigidpoints}.

 To establish Step 2, let us   define $\LKtrace{\phi}$.  For this, assume that $\alpha\in
 \operatorname{A}^i(X)$.  Then there is a smooth projective irreducible curve $T$
 over $k$,
 with $k$-points $t_1,t_0$,  and a cycle $Z\in \operatorname{CH}^i(T\times_k X)$
 such that $\alpha=w_{Z,t_0}(t_1)$ (e.g., \cite[Exa.~10.3.2]{fulton}).
 Since $\phi_{X_\Omega}$ is regular, there is an $\Omega$-morphism $\gamma:
 T_\Omega \to A$ that  induces $\phi_{X_\Omega}\circ
 w_{Z_\Omega}$.  The morphism $\gamma$ necessarily factors through the Albanese
 map $a_\Omega: T_\Omega \ra \alb(T_\Omega)$, via a morphism say
 $\delta:\operatorname{Alb}(T_\Omega)\to A$.  Since
 $(\alb(T_\Omega),a_\Omega) \iso (\alb(T),a)_\Omega$,    the definition of the
 $\Omega/k$-trace and rigidity (see \S \ref{S:rigid}) together imply  there is a 
 $k$-morphism $\epsilon:\operatorname{Alb}(T)\to  \LKtrace{A}$ so
 that $\delta=\tau \circ \epsilon_\Omega$.  In other words, there is a
 commutative diagram
 $$
 \xymatrix@R=.7cm{
  &T_\Omega(\Omega) \ar[r]^{w_{Z_{\Omega}}} \ar@/^2pc/@{->}[rr]^{\gamma(\Omega)}
  \ar[d]^{a_\Omega(\Omega)}& \operatorname{A}^i(X_\Omega) 
  \ar[r]^<>(0.5){\phi_{X_\Omega}}& A(\Omega)\\
  T(k)\ar[ru] \ar[d]_{a(k)}&\operatorname{Alb}(T_\Omega)(\Omega)
  \ar[rru]_{\delta(\Omega)} \ar[rr]_{\epsilon_\Omega(\Omega)}& & 
  \LKtrace{A}(\Omega) \ar[u]_{\tau(\Omega)} \\
  \operatorname{Alb}(T)(k) \ar[ru]
  \ar[rr]_{\epsilon(k)}&&\LKtrace{A}(k). \ar[ru]_{\text{ \ \ base
    change}}&\\
 }
 $$
 We define $\LKtrace{\phi}(\alpha):=\epsilon(a(t_1)-a(t_0))$.  Since the base
 change map on points is injective and since $\operatorname{Ker}( \tau)(\Omega)$
 is trivial (see
 \eqref{eq:rigidpoints}), $\LKtrace{\phi}(\alpha)$  is well-defined, i.e., it is 
 independent of the choice of $T$.
 It is  clear from our construction of $\LKtrace{\phi}$  that \eqref{E:LKphi}
 commutes.  In fact, starting with a general smooth pointed connected  variety
 $T$, the same argument shows that $\LKtrace{\phi}$ is regular (since $\epsilon
 \circ a$ is defined over $k$).

 \noindent  \textbf{Step 3.} \emph{In the notation of Step 2, if $\phi_{X_\Omega}$ is surjective, then $\LKtrace{\phi}:A^i(X)\to \LKtrace{A}(k)$ is surjective and $\tau:\LKtrace{A}_\Omega\to A$ is a purely inseparable  isogeny
  (and an isomorphism in  characteristic $0$).}
 
 \noindent Note that  by the criterion of \cite[Thm.~2.2]{hsaito}, \cite[Prop.~2.1]{murre83} (see Proposition  \ref{P:saito}), it follows from Step 3 and \eqref{E:LKphi} that  if $\operatorname{A}^i(X)$ has an algebraic representative, then so does $\operatorname{A}^i(X_\Omega)$.

 To establish Step 3, recall  that if $\phi_{X_\Omega}$ is surjective, then 
 $A/\Omega$ is dominated by an inverse system of abelian
 varieties defined over $k$ (see the discussion in \S \ref{S:chow})\,; in fact, 
 there is an abelian variety $B/k$ (namely the Albanese of a finite union of smooth
 alterations 
 of components of the Chow scheme)
 such that there is a surjection $\alpha:B_\Omega
 \to
 A$.    It follows (see \S \ref{S:rigid}) that $\tau$ is a purely inseparable  isogeny
 (and an isomorphism in  characteristic $0$).

 Let us now show that $\LKtrace{\phi}$ is surjective provided $\phi_{X_\Omega}$ is surjective. Let $\ell\in \mathbb N$ be a  prime $\ne \operatorname{char}(k)$, and let us consider $\ell$-torsion in the commutative diagram \eqref{E:LKphi}.
 The morphism $\tau$
 induces an isomorphism \eqref{eq:rigidpoints} $T_\ell \tau: T_\ell
 \LKtrace{A}(\Omega)\to T_\ell  A(\Omega)$. Lecomte's rigidity Theorem
 \ref{P:acbc} says that the homomorphism $T_\ell b$ is an isomorphism.
 The base change map $ T_\ell \LKtrace A(k) \to T_\ell \LKtrace
 A(\Omega)$ on torsion is also an isomorphism.  Therefore, a diagram
 chase in \eqref{E:LKphi} shows the surjectivity of $T_\ell
 \LKtrace{\phi}$, and thus that of $\LKtrace{\phi}$
 (\cite[Lem.~1.6.2(i)]{murre83}\,; see Lemma \ref{L:M-1.6.2}(a)).

 \noindent  \textbf{Step 4.}
 \emph{
  Let $(\operatorname{Ab}^i(X),\phi^i_{X})$ and  $(\operatorname{Ab}^i(X_{\Omega}),\phi^i_{X_\Omega})$  be  algebraic representatives  for   $\A^i(X)$ and
  $\A^i(X_\Omega)$, respectively.  The $k$-homomorphism   $b':\operatorname{Ab}^i(X)\to \LKtrace{\Ab}^i(X_\Omega)$ induced by the   regular homomorphism 
  $\LKtrace{\phi}:A^i(X)\to \LKtrace{\operatorname{Ab}}^i(X_\Omega)(k)$ of Step 2  
  is an isomorphism in characteristic zero, and a purely inseparable
  isogeny in positive characteristic.}
 
 \noindent  It suffices to show that $b'(k)$ is an
 isomorphism.
 We have the following 
 commutative diagram\,:
 \begin{equation}\label{E:1}
 \xymatrix@R=.5cm{
  \operatorname{A}^i(X) \ar@{->>}[r]^<>(0.5){\phi^i_X} \ar@{=}[d]&
  \operatorname{Ab}^i(X)(k) \ar[d]^{b'(k)}\\
  \operatorname{A}^i(X) \ar[r]^<>(0.5){\LKtrace{\phi}}&
  \LKtrace{\operatorname{Ab}}^i(X_\Omega)(k).
 }
 \end{equation}
 We already know from Step 3 that $\LKtrace{\phi}$ is surjective. Therefore, it remains only to prove that $b'$ is injective on $k$-points.
 Let $\zeta\in \operatorname{Ker} b'(k)$ and let $z\in
 \operatorname{A}^i(X)$ be a cycle mapping to $\zeta$ (i.e.,
 $\phi_X^i(z)=\zeta$).  Then $z_\Omega:=b(z) \in \operatorname{A}^i(X_\Omega)$ is abelian equivalent to $0$ (i.e., $z_\Omega$ is sent to $0$ under any regular morphism), since
 $\phi_{X_\Omega}^i(z_\Omega)=0$ by the commutativity of
 \eqref{E:LKphi} and \eqref{E:1}.  In  particular, considering the regular homomorphism  $(\phi^i_X)_\Omega$ of Step 1, we have 
 $(\phi^i_X)_\Omega (z_\Omega)=0$.  By virtue of the  commutativity of \eqref{E:ACBC3},
 and the injectivity of the base change map, we have
 that $\phi_X^i(z)=0$\,; i.e., $\zeta=0$, and we have that $b'$ is
 injective on geometric points.   
\end{proof}

\section{Algebraic representatives and Galois actions} \label{S:Gal}

In this section we extend the notion of algebraic representative to the setting
of Galois field extensions.  The main result of the section is Theorem
\ref{T:Desc}, which says that algebraic representatives descend from the
algebraic closure of a perfect field.  We also establish several basic results
on regular homomorphisms in this setting following \cite{murre83}.

Throughout this section, we take $k$ to be an 
algebraically closed field, and  $K \subseteq k$ to be a perfect sub-field such
that $k/K$ is algebraic.  (In particular, $K$ is the fixed field of $\gal(k/K)$.)

\subsection{Preliminaries on Galois descent}

If $X/k$ is a variety and $\sigma \in \gal(k/K)$, let $X^\sigma$ be
the fiber product $X^\sigma = X\times_{\spec k, \sigma} \spec k$\,; by definition, it comes equipped with a  map $\tilde\sigma: X^\sigma \to X$ over $\sigma$.   A
$k/K$ descent datum on $X$ is a system of  isomorphisms
$\st{\sigma_X: X \stackrel{\sim}{\to} X \text { over } \sigma: \sigma \in \gal(k/K)}$ satisfying
the co-cycle condition $(\sigma\tau)_X = \sigma_X \circ \tau_X: X  \to X$.  
There is a similar notion
of a descent datum on a morphism of varieties (i.e., a descent datum
on its graph), and the category of quasi-projective $k$-schemes
equipped with $k/K$ descent data is equivalent to the category of
quasi-projective schemes over $K$ (e.g., \cite[Thm.~6, p.135, p.141]{BLR}, 
\cite[Thm.~14.84, p.457]{GW10}).  We will often denote the effect of
this functor by $(X, \st{\sigma}) \mapsto \underline X$\,; the inverse
functor is $\underline Y \mapsto Y := \underline Y \times_{\spec
 K}\spec k$.  Similar notation will be employed for cycles.

Note that a $k/K$ descent datum on $X$ (or, equivalently, a choice of
$K$-model $\underline X$) determines an action of $\gal(k/K)$ on the $k$-points 
$X(k)$\,; a morphism $X \to Y$ of $k$-varieties with descent data
descends to $K$ if and only if its action on $k$-points is
$\gal(k/K)$-equivariant (e.g., \cite[Prop.~16.9]{milneAG}).

\subsection{Action of the Galois group on cycles}

Let $X/k$ be a quasi-projective variety equipped with a $k/K$-descent
datum $\st{\sigma_X}$.  Pullback by the flat morphism of schemes
$\sigma_X$ induces \cite[\S 1.7]{fulton} a
homomorphism
$$
\sigma_X^*:\operatorname{CH}^i(X)\to \operatorname{CH}^i(X),
$$
and we have $\sigma_X^*[Z] = [\sigma_X^*Z]$.
The co-cycle condition for descent data then yields an action of
$\gal(k/K)$ on $\chow^i(X)$, which restricts to an action on
$\A^i(X)$.

Now let $T/k$ be a smooth, connected variety over $k$\,; let $Z$ be a
codimension $i$ cycle on $T\times_k X$, and let $t_0: \spec k \to T$
be a $k$-point. 
If $T$, $Z$ and $t_0$ are equipped with compatible $k/K$ descent data, we obtain a commutative diagram 
$$
\xymatrix@C=1.5cm@R=.75cm{
 T(k) \ar@{->}[r]^{w_{Z,t_0}}  \ar@{->}[d]_{ \sigma_T^*} &   \operatorname{A}^i(X)
 \ar@{->}[d]^{\sigma_X^*}   \\
 T(k)  \ar@{->}[r]^{w_{Z,t_0}} &   \operatorname{A}^i(X). \\}
$$
and we see that $w_Z$ is $\gal(k/K)$-equivariant.  

Even in the absence of a distinguished $K$-model for $T$ and $Z$ one can show, using the refined Gysin map \cite[\S 6.3]{fulton}, that 
\begin{equation}\label{E:CycFiber}
\sigma_X^* Z_{t_0}=((\tilde \sigma \times_\sigma \sigma_X)^*Z)_{\tilde
 \sigma^*t_0}
\end{equation}
and we thus have a commutative diagram
\begin{equation}\label{E:WZTsig}
\xymatrix@C=2.5cm@R=.75cm{
 T(k) \ar@{->}[r]^{w_{Z,t_0}}  \ar@{->}[d]_{\tilde \sigma^*} &  
 \operatorname{A}^i(X) \ar@{->}[d]^{\sigma_X^*}   \\
 T^\sigma(k)  \ar@{->}[r]^{w_{(\tilde \sigma \times_\sigma \sigma_X)^*Z,\tilde
   \sigma^*t_0}} &   \operatorname{A}^i(X). \\}
\end{equation}

\subsection{Galois equivariant algebraic representatives}  \label{S:GEAR}
Let $k$ be an  algebraically closed field, and  $K \subseteq k$  a perfect
sub-field such that $k/K$ is algebraic.  Let  $X$   be a smooth projective
variety over $k$,  equipped with  $k/K$ descent data.  
Let $A$ be an abelian variety over $k$ that is also equipped with $k/K$  
descent data.  We say that a regular homomorphism 
$$
\phi:\operatorname{A}^i(X)\to A(k)
$$
is \emph{$\operatorname{Gal}(k/K)$-equivariant} if for each $\sigma\in
\operatorname{Gal}(k/K)$, the following diagram is commutative\,:
$$
\xymatrix@C=1.5cm@R=.75cm{
 \operatorname{A}^i(X) \ar@{->}[r]^{\phi} \ar@{->}[d]^{\sigma_X^{*}}  & A(k)
 \ar@{->}[d]^{\sigma_{A}^*}\\
 \operatorname{A}^i(X) \ar@{->}[r]^{\phi}& A(k). \\
}
$$

\begin{rem}
 \label{R:descendabvar}
 Given such data $(A,\st{\sigma_A}, \phi)$, the class $[0] \in \A^i(X)$ is fixed by $\gal(k/K)$, and so $\phi([0])$ descends to $K$\,; $A(K)$ is nonempty. 
 By uniqueness of group structures on abelian varieties,
 we see that the group law on $A$ also descends to $\underline A$\,; $A$ descends to $K$ as an abelian variety.
\end{rem}

\begin{dfn}\label{D:GEAR}  Let $k$ be an  algebraically closed field, and  $K
 \subseteq k$  a
 perfect sub-field such that $k/K$ is algebraic.  Let  $X$   be a  smooth
 projective variety over $k$,  equipped with  $k/K$ descent data.  
 A \emph{Galois equivariant  algebraic  representative} for
 $\operatorname{A}^i(X)$  is a pair
 $(\Ab^i_{\gal(k/K)}(X),\phi_{X,\gal(k/K)})$ which is initial among all pairs
 $(A,\phi)$ consisting of an abelian variety $A/k$ with $k/K$ descent
 datum and a Galois-equivariant regular homomorphism $\phi:\A^i(X) \to
 A(k)$.
\end{dfn}

\begin{rem}\label{R:desendmorphismabelian}
 Note that for a pointed smooth integral 
 $k$-variety $T$, and a codimension
 $i$ cycle
 $Z$ on $T\times_kX$, the following diagram is commutative for a Galois
 equivariant regular homomorphism $\phi:\operatorname{A}^i(X)\to A(k)$\,:
 $$
 \xymatrix@C=2.5cm@R=.75cm{
  T(k) \ar@{->}[r]^{w_{Z,t_0}}  \ar@{->}[d]^{\tilde \sigma^*} &  
  \operatorname{A}^i(X) \ar@{->}[r]^{\phi} \ar@{->}[d]^{\sigma_X^{*}}  & A(k)
  \ar@{->}[d]^{\sigma_{A}^*}\\
  T^\sigma(k)  \ar@{->}[r]^{w_{(\tilde \sigma \times_\sigma  \sigma_X)^*Z, \tilde
    \sigma ^*t_0}} &\operatorname{A}^i(X) \ar@{->}[r]^{\phi}& A(k). \\
 }
 $$
 Moreover,  for a smooth integral 
 $k$-variety $T$ with descent datum,
 together  with a
 $\operatorname{Gal}(k/K)$-invariant marked point $t_0$, and a codimension $i$
 cycle $Z$ on $T\times_kX$ such that the descent data for $T$ and $X$ induces
 descent data for $Z$, the following diagram is commutative for a Galois
 invariant regular homomorphism $\phi:\operatorname{A}^i(X)\to A(k)$\,:
 $$
 \xymatrix@C=1.5cm@R=.75cm{
  T(k) \ar@{->}[r]^{w_{Z,t_0}}  \ar@{->}[d]^{\sigma_T^*} &   \operatorname{A}^i(X)
  \ar@{->}[r]^{\phi} \ar@{->}[d]^{\sigma_X^{*}}  & A(k)
  \ar@{->}[d]^{\sigma_{A}^*}\\
  T(k)  \ar@{->}[r]^{w_{Z,t_0}} &\operatorname{A}^i(X) \ar@{->}[r]^{\phi}& A(k).
  \\
 }
 $$
 This says that the morphism of $k$-varieties $T\to A$ inducing
 $\phi\circ w_Z$ descends to $K$.
\end{rem}

We now prove some basic results about Galois equivariant regular morphisms.   
The main result we want is that for smooth projective varieties over a perfect
field $K$, algebraic representatives of the variety over the algebraic closure
give Galois equivariant algebraic representatives\,; in particular, they descend
to $K$.

\begin{teo}\label{T:Desc}
 Let $X/k$ be a smooth projective variety over an algebraically 
 closed field k, and suppose that $\operatorname{A}^i(X)$ has an 
 algebraic representative $(\operatorname{Ab}^i(X),\phi^i_X)$. Let 
 $K \subseteq k$ be a perfect field such that $k/K$ is algebraic. Then 
 each $k/K$ descent datum on $X$ induces a $k/K$ descent datum 
 on $\operatorname{Ab}^i(X)$. Moreover, with respect to these descent data,
 $\phi^i_X$ is $
 \operatorname{Gal}(k/K)$-equivariant, and the pair $
 (\operatorname{Ab}^i(X),\phi^i_X)$ is a  Galois 
 equivariant  algebraic representative.  
\end{teo}

\begin{rem}\label{R:DescAlgRep}
 Note that by  Remark \ref{R:descendabvar},  $\operatorname{Ab}^i(X)$ descends to an abelian variety $\underline\Ab^i(X)$  over $K$.
\end{rem}

The proof uses the universal property of $\Ab^i(X)$ to transfer
descent data for $X$ to descent data for $\Ab^i(X)$.    Before giving the proof,
we introduce one more notion that we will utilize.  

\begin{dfn} 
  Let $X$ be a smooth projective variety over an algebraically closed field $k$, and let $\sigma\in \operatorname{Aut}(k)$. 
 Given an abelian variety  $A$ over $k$, a  homomorphism of groups
 $$
 \begin{CD}
 \operatorname{A}^i(X)@>\phi>> A(k)
 \end{CD}
 $$
 is said to be \emph{$\sigma$-regular} if for every pair $(T,Z)$ with $T$ a 
 pointed smooth integral 
 $k$-variety, and $Z\subseteq T\times_k X$ a relative
 codimension $i$ cycle
 over $T$,  
 the composition 
 $$
 \begin{CD}
 T(k)@> w_Z >> \operatorname{A}^i(X)@>\phi >>A(k)
 \end{CD}
 $$ 
 is induced by a morphism $g:T\to A$ over $\sigma$\,;
 by this (i.e., $\phi\circ w_Z$ being induced by the morphism  $g$ over $\sigma$) we  mean there is  a commutative diagram of morphisms
 of schemes
 $$
 \xymatrix@R=1em{
 T\ar[r]^g \ar[d]& A \ar[d]\\
 k\ar[r]^\sigma &k
 }
 $$
 such that   for each $k$-point $t:\operatorname{Spec}k\to T$, we have
 $(\phi\circ w_Z)(t)=g_*(t):=g\circ t \circ \sigma^{-1}$.
\end{dfn}

\begin{lem}\label{L:SigReg}  Let $X$ be a smooth projective variety over an algebraically closed field $k$, and let $\sigma\in \operatorname{Aut}(k)$.  Suppose that
 $(\operatorname{Ab}^i(X),\phi^i_X)$ is
 an algebraic
 representative for $\operatorname{A}^i(X)$.
  Given an abelian variety $A/k$ and a $\sigma$-regular homomorphism
 $\phi:\operatorname{A}^i(X)\to A(k)$, there
 exists  a unique morphism $f:\operatorname{Ab}^i(X)\to A$ over $\sigma$   such 
 that $\phi=f_*\circ \phi_X^i$.
\end{lem}

\begin{proof} Let  $\phi:\operatorname{A}^i(X)\to A(k)$ be a $\sigma$-regular
 homomorphism.  
 Consider the morphisms
 $$
 \begin{CD}
 A(k)@>\tilde \sigma ^*>> A^\sigma(k), @. \ \ \ \ \  @. A@>\tilde \sigma ^{-1}>>
 A^\sigma,\\
 \end{CD}
 $$
 with the latter over $\sigma^{-1}\in \operatorname{Aut}(k)$, and where  given a $k$-point $a:\operatorname{Spec} k\to A$, we define  $\tilde \sigma^*(a):=(\tilde \sigma^{-1})_*(a)=\tilde \sigma^{-1} a \sigma$.    
 
 The composition $\tilde \sigma^*\circ \phi:\operatorname{A}^i(X)\to A^\sigma(k)$ is a regular homomorphism.
 Indeed,  given a triple  $(T,Z,t_0)$ with $T$ a smooth integral   $k$-variety,
 $Z$ a   codimension $i$ cycle on 
 $ T\times_k X$, and $t_0:\operatorname{Spec} k\to T$ a $k$-point, there is by
 definition a morphism $g:T\to A$ over $\sigma$ such that for each $k$-point
 $t:\operatorname{Spec}k \to T$ we have $(\phi  w_Z)(t)= gt\sigma^{-1}$.  
 Therefore, $(\tilde \sigma^* \phi w_Z)(t)=\tilde \sigma^*
 (gt\sigma^{-1})=\tilde\sigma^{-1}(gt\sigma^{-1})\sigma=\tilde \sigma^{-1}gt$, 
 so that $\tilde \sigma^*\phi w_Z$ is induced by  $\tilde \sigma^{-1}g:T\to A^\sigma$, and we have
 established that  $\tilde \sigma^*\phi$ is a regular homomorphism.

 It follows from  the universal property of the algebraic representative that 
 there is a unique morphism $\tilde f:\operatorname{Ab}^i(X)\to A^\sigma$ of
 abelian varieties over $k$, such that $\tilde \sigma^*\circ \phi=\tilde
 f\circ \phi^i_X$ on $k$-points.   The morphism $f:=\tilde \sigma\circ \tilde f$
 over $\sigma$
 satisfies $\phi=\tilde \sigma_*\tilde\sigma^* \phi= 
 \tilde \sigma_*\tilde f \phi^i_X=
 f_*\phi^i_X$.
Here we are using that $\tilde \sigma_*\tilde f = f_*$ on $k$-points\,; indeed, given a $k$-point  $a:\operatorname{Spec}k\to \operatorname{Ab}^i(X)$, we have $\tilde \sigma_*\tilde f (a)=\tilde \sigma_*(\tilde fa)=\tilde \sigma \tilde fa\sigma^{-1}=fa\sigma^{-1}=f_*(a)$.  
 The uniqueness
 of $f$ follows  
 again  from the universal property of the algebraic representative. 
\end{proof}

\begin{proof}[Proof of Theorem \ref{T:Desc}] 
 We start by showing that  $k/K$ descent data for $X$
 induces  $k/K$ descent data for $\operatorname{Ab}^i(X)$.
 The first claim is that  the composition
 $\phi_X^i\circ\sigma_X^*:\operatorname{A}^i(X)\to \operatorname{Ab}^i(X)(k)$ is 
 $\sigma^{-1}$-regular.   
 Indeed,  suppose we are given $T/k$, a smooth integral 
 pointed variety,  and  $Z$  a  codimension $i$ cycle
 on $T\times_k X$.
 It was shown in \eqref{E:WZTsig} that we have  a commutative diagram 
 $$
 \xymatrix@C=1.5cm@R=.75cm{
  T(k) \ar@{->}[r]^{w_Z}  \ar@{->}[d]^{\tilde \sigma^*} &   \operatorname{A}^i(X)
  \ar@{->}[r]^{\phi^i_X} \ar@{->}[d]^{\sigma_X^*}  & \operatorname{Ab}^i(X)(k) \\
  T^\sigma(k)  \ar@{->}[r]^{w_{(\tilde \sigma \times \sigma_X)^*Z}} &  
  \operatorname{A}^i(X) \ar@{->}[r]^{\phi^i_X}&  \operatorname{Ab}^i(X)(k). \\
 }
 $$
 Let $g':T^\sigma \to \operatorname{Ab}^i(X)$ be the morphism of $k$-varieties
 inducing the bottom row of the diagram, and set $g:=g'\circ \tilde
 \sigma^{-1}:T\to \operatorname{Ab}^i(X)$ to be the induced morphism over
 $\sigma^{-1}$.   For a $k$-point $t:\operatorname{Spec} k \to T$  we have
 $$
 (\phi^i_X \sigma_X^*w_Z)(t)=(\phi^i_X w_{(\tilde \sigma \times \sigma_X)^*Z}
 \tilde \sigma^*)(t)=g'(\tilde \sigma^*t)=g'\tilde \sigma^{-1}t\sigma
 =gt\sigma=g_*(t).
 $$  
 Therefore we have established that $\phi^i_X\sigma_X^*$ is
 $\sigma^{-1}$-regular.

 By virtue of  the previous lemma, there is a unique  morphism
 $(\sigma_{\operatorname{Ab}})^{-1}:\operatorname{Ab}^i(X)\to
 \operatorname{Ab}^i(X)$
 over $\sigma^{-1}$ making the following diagram commute\,:
 $$
 \xymatrix@C=1.5cm@R=.75cm{
  T(k) \ar@{->}[r]^{w_Z}  \ar@{->}[d]^{\tilde \sigma^*} &   \operatorname{A}^i(X)
  \ar@{->}[r]^{\phi^i_X} \ar@{->}[d]^{\sigma_X^*}  & \operatorname{Ab}^i(X)(k)
  \ar@{->}[d]^{((\sigma_{\operatorname{Ab}})^{-1})_*}\\
  T^\sigma(k)  \ar@{->}[r]^{w_{(\tilde \sigma \times \sigma_X)^*Z}} &  
  \operatorname{A}^i(X) \ar@{->}[r]^{\phi^i_X}&  \operatorname{Ab}^i(X)(k). \\
 }
 $$
 
 Using the universal property of the algebraic representative, one can
 see that the morphisms
 $(\sigma_{\operatorname{Ab}})^{-1}:\operatorname{Ab}^i(X)\to
 \operatorname{Ab}^i(X)$ are isomorphisms, and that the morphisms 
 $\sigma_{\operatorname{Ab}}:=((\sigma_{\operatorname{Ab}})^{-1})^{-1}$ 
 define a lift of the
 action of $\operatorname{Gal}(k/K)$ to $\operatorname{Ab}^i(X)$.
 In particular, the $\sigma_{\operatorname{Ab}}$ provide a  $k/K$ descent datum
 for $\operatorname{Ab}^i(X)$.
 Finally, the commutativity of the diagram
 $$
 \xymatrix@C=1.5cm@R=.75cm{
  T(k) \ar@{->}[r]^{w_Z}  \ar@{->}[d]^{\tilde \sigma^*} &   \operatorname{A}^i(X)
  \ar@{->}[r]^{\phi^i_X} \ar@{->}[d]^{\sigma_X^*}  & \operatorname{Ab}^i(X)(k)
  \ar@{->}[d]^{\sigma_{\operatorname{Ab}}^*}\\
  T^\sigma(k)  \ar@{->}[r]^{w_{(\tilde \sigma \times \sigma_X)^*Z}} &  
  \operatorname{A}^i(X) \ar@{->}[r]^{\phi^i_X}&  \operatorname{Ab}^i(X)(k). \\
 }
 $$
 shows that $\phi^i_X$ is $\operatorname{Gal}(k/K)$-equivariant, and therefore that 
 $(\operatorname{Ab}^i(X),\phi^i_X)$ is a Galois equivariant  algebraic
 representative.   
\end{proof}

\subsection{Some structure results for algebraic representatives}

We would also like to extend some basic structure results from Murre
\cite{murre83} to the case of Galois equivariant regular homomorphisms\,; a crucial ingredient is the following
technical, yet general, proposition  extending a well known result of Weil  \cite[Lem.~9]{weil54}  over an algebraically closed field
to the setting of 
arbitrary perfect fields\,:

\begin{pro}[\cite{ACMVabtriv}] \label{P:algcycles}
 Let $\underline X/K$ be a  scheme of finite type   
 over a 
 perfect
 field $K$ and
 let $X$ be the base-change of $\underline X$ to an algebraic closure $k$ of $K$.
 If $\alpha \in \operatorname{CH}^i(X)$
 is algebraically trivial, then there exist
 an abelian variety $\underline A/K$,  a cycle $\underline  Z$ on $\underline A\cross_K  \underline
 X$, and a pair of
 $k$-points $t_1$, $t_0$ on
 $\underline A$   such that $\alpha = Z_{t_1} - Z_{t_0}$.\qed
\end{pro}

\begin{lem}[{\cite[Lem.~1.6.2, Cor.~1.6.3]{murre83}}] \label{L:M-1.6.2} 
Let  $k$ be an algebraically closed field, let $K \subseteq k$ be a perfect sub-field such that $k/K$ is algebraic, let $X$ be a smooth projective variety over $k$, let $A$ be an abelian variety over $k$,  and assume that each is equipped with $k/K$-descent data.  
Now let  $\phi: \operatorname{A}^i(X)\to  A(k)$ be a   $\operatorname{Gal}(k/K)$-equivariant regular homomorphism.  Then\,:
 
 \begin{alphabetize} 
  \item There is an abelian subvariety $\underline A'\subseteq \underline A$
      defined over $K$ such that
  $\operatorname{Im}(\phi)=A'(k)$.  
  \item There is an abelian variety $\underline D$ defined over $K$ and a cycle
  $\underline Z\in
  \operatorname{CH}^i(\underline D\times_K \underline X)$ such that the
  composition $\begin{CD}
  D(k)@> w_{Z}>> \operatorname{A}^i(X) @>\phi >>A(k)\end{CD}$ surjects onto
  $\operatorname{Im}(\phi)$ and  is induced by a
  $K$-morphism of
  abelian varieties $\underline D\to \underline A$.  
  
  \item  If $\phi$ is surjective, then there is  an abelian variety $\underline B$
  defined over $K$ and a cycle $\underline Z\in
  \operatorname{CH}^i(\underline B\times_K \underline X)$ such that the
  composition $\begin{CD}
  B(k)@> w_{Z}>> \operatorname{A}^i(X) @>\phi >>
  A(k) \end{CD}$ is induced by  a $K$-isogeny of abelian
  varieties  $\underline B\to \underline A$.

  \item 
  If $\phi$ is surjective, then  there exists   a cycle $\underline Z\in
  \operatorname{CH}^i(\underline A\times_K \underline  X)$ such that
  the composition $$\begin{CD}
  A(k)@> w_{Z}>> \operatorname{A}^i(X) @>\phi >>
  A(k) \end{CD}$$ is induced by the  $K$-morphism  $r\cdot
  \operatorname{Id}_{\underline A}:\underline A\to \underline A$, for some
  non-zero integer $r$.   
 \end{alphabetize}
\end{lem}

\begin{proof} This is  
 \cite[Lem.~1.6.2, Cor.~1.6.3]{murre83} in the case $K=k$. 
 We will see that, with the arithmetic input of Proposition \ref{P:algcycles}, the same
 proof establishes the present lemma.
 As (a) follows from the proof of (b), we begin by proving (b).  
 
 (b)  The key point is that given $\alpha\in \operatorname{Im}(\phi)$,   there
 exist an abelian variety $\underline D$ defined over $K$ and a
 $\underline Y\in \operatorname{CH}^i(\underline  D\times_K \underline X)$ such
 that $\alpha\in
 \operatorname{Im}(\phi\circ w_{Y})$.   In the case $k=K$, this is
 \cite[Lem.~9]{weil54}\,;  in the case $k\ne K$, this is Proposition
 \ref{P:algcycles}.
 
 Now for all such pairs $(\underline  D, \underline  Y)$ consider the abelian
 sub-varieties
 $\operatorname{Im}(\phi\circ w_{Y})$. Note that, since we are assuming that $\phi$ is $\gal(k/K)$-equivariant, Remark \ref{R:desendmorphismabelian} says that $\phi\circ w_{Y}$ descends to $K$.
Clearly we can take a pair
 $(\underline  D',\underline  Y')$ such that $\operatorname{Im}(\phi\circ
 w_{Y'})$ is of maximal
 dimension.  The claim is that   $\operatorname{Im}(\phi\circ
 w_{Y'})=\operatorname{Im}(\phi)$.  
 If not, let $\phi(\alpha)\notin \operatorname{Im}(\phi\circ w_{Y'})$,
 then by the assertion of the first paragraph, there exists a pair  $(\underline
 D'',\underline  Y'')$
 such that $\phi(\alpha)\in \operatorname{Im}(\phi\circ w_{Y''})$.  Now
 consider $\underline  D''':=\underline  D'\times_K \underline  D''$ and the
 projections $\pi':\underline  D'\times_K\underline  D''\times_K\underline  X
 \to \underline  D'\times_K\underline  X$ and
 $\pi'':\underline  D'\times_K\underline  D''\times_K\underline  X \to \underline
 D''\times_K\underline  X$.  Take 
 $$
 \underline  Y''':=\pi'^*\underline  Y'+\pi''^*\underline  Y''.
 $$
 Then $\operatorname{Im}(\phi\circ w_{Y'''})\supset
 \operatorname{Im}(\phi\circ w_{Y'})$ and $\operatorname{Im}(\phi\circ
 w_{Y'''})\supset \operatorname{Im}(\phi\circ w_{Y''})$, a
 contradiction, proving the claim, and thus (b).     
 
 (c)  We have a surjection $\underline  D\to \underline  A$, and there exists a
 $K$-sub-abelian variety $\underline  B\subseteq \underline  D$ such that the
 induced map $\underline  B\to \underline  A$ is an isogeny.  The restricted
 cycle $\underline  Z|_{\underline  B\times_K\underline   X}$ provides the
 desired cycle.

 (d) This follows from (c).   The dual  isogeny
 $g:\underline  A\to \underline  B$ has the property that there is a non-zero
 integer $r\in \mathbb Z$ such that the composition
 $\underline  A\to \underline  B\to
 \underline  A$ is given by $r\cdot
 \operatorname{Id}_{\underline  A}$.    The cycle
 $(g\times_K\operatorname{Id}_{\underline  X})^*\underline  Z$ gives the desired
 cycle.    
\end{proof}

In light of Murre's results \cite{murre83}  (see Theorem \ref{T:Murre}), we
have the following consequence (see also Corollary \ref{C:cycles3} and Remark \ref{R:corr}, which establishes 
a
stronger statement  for cohomology with $\mathbb Q_\ell$-coefficients, under the
hypothesis $K\subseteq \mathbb C$).

\begin{cor} \label{C:Murre}  
Let $X$ be a smooth projective variety over a perfect field $K$.
 Let $\underline\Ab^2(X)/K$ be the model   of $\Ab^2(X)$ from Theorem \ref{T:Desc}.
There is a
 natural inclusion of $\operatorname{Gal}(k/K)$-representations 
 $$
 H^1(\widehat {\underline{\Ab}^2(X)}_k,\rat_\ell) \hookrightarrow
 H^3(X,\rat_\ell(1)).
 $$
 For almost all $\ell$ there is an inclusion
 $H^1(\widehat{\underline{\Ab}^2(X)}_k,\mathbb Z_\ell) \hookrightarrow
 H^3(X,\mathbb Z_\ell(1))$ of $\operatorname{Gal}(k/K)$-modules.
\end{cor}

\begin{proof}
 We aim to show that the two inclusions \eqref{E:degWZ}, i.e.,
 $T_\ell\Ab^2(X)\hookrightarrow 
 T_\ell\operatorname{A}^2(X)$,   and \eqref{E:ChtoH3}, i.e.,  $T_\ell
 \operatorname{CH}^2(X)\hookrightarrow H^{3}(X,\mathbb
 Z_\ell(2))
 $ 
 constructed by Murre are $\operatorname{Gal}(k/K)$-equivariant.
 The second inclusion  \eqref{E:ChtoH3} can be seen to be   equivariant since the
 map is constructed via natural maps of sheaves, all of which have natural
 $\operatorname{Gal}(k/K)$-actions.
 Let us now show that the first inclusion \eqref{E:degWZ} can be taken to be  
 equivariant. 
 Using the modification of \cite[Cor.~1.6.3]{murre83} given in Lemma
 \ref{L:M-1.6.2},  
 there  is a cycle  $\underline Z\in \operatorname{CH}^2(
 \underline{\operatorname{Ab}}^2(X) \times_{K}\underline X)$ such that 
 $$
 \phi^2_{X} \circ w_{Z}:\operatorname{Ab}^2(X)(k)\to \operatorname{A}^2(X)\to
 \operatorname{Ab}^2(X)(k)
 $$
 is induced by a $K$-isogeny
 $ f:\underline{\operatorname{Ab}}^2(X)\to 
 \underline{\operatorname{Ab}}^2(X)$.   We also saw in the proof of Theorem 
 \ref{T:Desc}
 that both $w_{Z}$ and $\phi^2_{X}$ are Galois equivariant. 
 Now take a prime number $\ell$ such that $(\ell,\deg (f))=1$ and $\ell$ is invertible in $k$\,;
 on points of order $\ell^\nu$ we have
 $$
 f :
 {\operatorname{Ab}}^2(X)[\ell^\nu]\stackrel{\sim}{\to}{\operatorname{Ab}}^2(X)[\ell^\nu],
 $$
 as Galois representations,
 for all $\nu>0$.  Therefore 
 ${\operatorname{Ab}}^2(X)[\ell^\nu]$ is a direct summand of the Galois representation
 $\operatorname{A}^2(X)[\ell^\nu]$ and we obtain a Galois equivariant inclusion
 $$
 {\operatorname{Ab}}^2(X)[\ell^\nu] \hookrightarrow
 \operatorname{A}^2(X)[\ell^\nu].
 $$ 
 Taking the associated Tate modules completes the proof. 
\end{proof}

\section{Phantoms via algebraic representatives} \label{S:pThAII}

We now prove Theorem \ref{T:jacdescent} and complete the proof of Theorem \ref{T:cycles}, by showing in Theorem \ref{T:cycles3'} that the image of the Abel--Jacobi map in the  intermediate Jacobian $J^3(X_\cx)$ descends to a field of definition of $X$.

\begin{teo}\label{T:cycles3'}
 Suppose  $X$ is a smooth 
 projective variety over a field $K\subseteq \mathbb C$,
 and  $n$ is a non-negative integer.  Assume  that
 $\operatorname{A}^{n+1}(X_{\mathbb C})$ 
 admits  $(J^{2n+1}_a(X_{\mathbb C}),AJ)$ as an algebraic representative,
 where 
 $J^{2n+1}_a(X_{\mathbb C})$ is the
 image of the Abel--Jacobi map $AJ:\operatorname{A}^{n+1}(X_{\mathbb C})\to
 J^{2n+1}(X_{\mathbb C})$. 
 Then  $J^{2n+1}_a(X_{\mathbb C})$ has a distinguished model $J$ over $K$  
 making $AJ:\operatorname{A}^{n+1}(X_{\mathbb C})\to
 J_a^{2n+1}(X_{\mathbb C})$ an $\operatorname{Aut}(\cx/K)$-equivariant homomorphism,
 and
 there is a correspondence $\gamma$ on $\widehat J\times_K X$ inducing for each 
 prime
 number $\ell$  an  inclusion of
 $\operatorname{Gal}(K)$-representations 
 \begin{equation}\label{E:cycles3'}
 H^1(\widehat J_{\bar K},\mathbb Q_\ell)\stackrel{\gamma_*}{\hookrightarrow}
 H^{2n+1}(X_{\bar K},\mathbb Q_\ell(n)),
 \end{equation}
 with image $\coniveau ^n H^{2n+1}(X_{\bar K},\mathbb Q_\ell(n))$.  Consequently,
 if in addition $H^{2n+1}(X_{\mathbb C},\mathbb Q)$ is of geometric  coniveau
 $n$, then \eqref{E:cycles3'} is an isomorphism.  
\end{teo}

\begin{proof} 
Let  $X$ be a smooth  projective
variety  defined over
$K\subseteq \mathbb C$ and assume that $\operatorname{A}^{n+1}(X_{\mathbb C})$ admits an algebraic
representative $(\operatorname{Ab}^{n+1}(X_{\mathbb
	C}),\phi^{n+1}_{X_{\mathbb C}})$.
We have seen in Theorem \ref{T:ACBC}
that $\operatorname{A}^{n+1}(X_{\bar
	K})$  admits an algebraic representative $(\operatorname{Ab}^{n+1}(X_{\bar
	K}),\phi^{n+1}_{X_{\bar K}})$ with $\operatorname{Ab}^{n+1}(X_{\bar K})_{\mathbb C}\cong
\operatorname{Ab}^{n+1}(X_{\mathbb
	C})$, and in Theorem  \ref{T:Desc} that $\operatorname{Ab}^{n+1}(X_{\bar K})$
admits a distinguished model $J$ over $K$. (Note that the distinguished model $J$ over $K$ is indeed an abelian variety\,; see Remark \ref{R:descendabvar}.)  
We also saw in Theorem \ref{T:ACBC}
that  $(\operatorname{Ab}^{n+1}(X_{\bar
	K}),\phi^{n+1}_{X_{\bar K}})$ is the $\cx/\bar K$-trace of $(\operatorname{Ab}^{n+1}(X_{\mathbb
	C}),\phi^{n+1}_{X_{\mathbb C}})$ and in Theorem  \ref{T:Desc} that $\operatorname{Ab}^{n+1}(X_{\bar
	K})$ is equipped with Galois descent datum making $\phi^{n+1}_{X_{\bar K}}$ Galois-equivariant. Since the $\cx/\bar K$-trace is constructed by
	fpqc descent \cite{conradtrace},
	 we see that $\phi^{n+1}_{X_{\mathbb C}} : \operatorname{A}^{n+1}(X_{\mathbb C}) \to \operatorname{Ab}^{n+1}(X_{\mathbb
	C})$ is $\operatorname{Aut}(\cx/K)$-equivariant. Moreover, since the regular homomorphism $\phi^{n+1}_{X_{\mathbb C}} $ is surjective, the model $J$ is \emph{distinguished} in the sense that it is uniquely
 determined by the natural Galois action on $\operatorname{A}^{n+1}(X_\cx)$.

 We now make a general observation.  
 Let $X_{\mathbb C}$ be a
 smooth projective
 variety defined over $\mathbb C$.  
 Given an abelian variety $A_{\mathbb C}$ over $\mathbb C$ and a   correspondence $Z_{\mathbb C}\in \operatorname{CH}^{n+1}(A_{\mathbb C} \times_{\mathbb C} X_{\mathbb
  C})$, then using the functoriality of the Abel--Jacobi map with respect to the action of correspondences
 we have a commutative diagram 
 \begin{equation}\label{E:Z*diag}
 \xymatrix{
  A_{\mathbb C}(\mathbb C) \ar[r] \ar@{->>}[rd] \ar@/^1.5pc/@{->}[rr]^{w_{Z_{\mathbb C}}} &
  \operatorname{A}_0(A_{\mathbb C})
  \ar[r]^<>(.5){Z_*} \ar[d]^{AJ}&  \operatorname{A}^{n+1}(X_{\mathbb C})
  \ar[d]^{AJ}\\
  &\operatorname{Alb}(A_{\mathbb C})(\mathbb C) \ar@{->}[r]& J^{2n+1}(X_{\mathbb C})(\mathbb
  C).\\
 }
 \end{equation}
 The differential of the map  $\operatorname{Alb}(A_{\mathbb C})\to J^{2n+1}(X_{\mathbb C})$
 in the bottom row  is the complexification of the 
 map
 \begin{equation}\label{E:Z*map}
 Z_*:H^{2\dim A-1}(A_{\mathbb C},\mathbb Q(\dim A_{\mathbb C} -1))\to H^{2n+1}(X_{\mathbb C},\mathbb
 Q(n)),
 \end{equation}
 induced by the correspondence $Z_{\mathbb C}$ (e.g., \cite[Thm.~12.17]{voisinI}).

 The natural identification $H^1(\widehat A_{\mathbb C},\mathbb Q)\to H^{2\dim A-1}(A_{\mathbb C},\mathbb
 Q(\dim A_{\mathbb C}-1))$ can be obtained via correspondences as the composition of
 isomorphisms $\Theta_* \circ \Gamma_{\lambda *}$, where $\Gamma_\lambda$ is the
 graph of a polarization   $\lambda:\widehat A_{\mathbb C}\to
 A_{\mathbb C}$, and 
 $\Theta_* : H^1(A_{\mathbb C},\mathbb Q)\to  H^{2\dim A-1}(A_{\mathbb C},\mathbb Q(\dim A_{\mathbb C} -1))$ is
 induced
 by the $(\dim A_{\mathbb C}-1)$-fold intersection $\Theta$ of  the
 class of the dual  polarization on $A_{\mathbb C}$.
 In conclusion, we have a morphism 
 \begin{equation}\label{E:HdgMor}
 \begin{CD}
 H^1(\widehat A_{\mathbb C},\mathbb Q) @>Z_*\circ \Theta_*\circ \Gamma_{\lambda *}>>
 H^{2n+1}(X_{\mathbb C},\mathbb Q(n)).
 \end{CD}
 \end{equation}
 Finally consider the special case where  $(A_{\mathbb C},\phi)=(J_a^{2n+1},AJ)$, and where the cycle  $Z_{\mathbb C}$ is taken as in  \cite[Cor.~1.6.3]{murre83} (see Lemma \ref{L:M-1.6.2}(d)) so that $AJ\circ w_{Z_{\mathbb C}}$ is given by $r\cdot \operatorname{Id}_{J_a^{2n+1}}$ for some non-zero integer $r$. 
 The diagram \eqref{E:Z*diag}   then shows that \eqref{E:HdgMor} is injective, with image $\coniveau ^nH^{2n+1}(X_{\mathbb C},\mathbb Q)$.

  With this set-up, Lemma \ref{L:M-1.6.2} implies that there exists   a cycle $Z\in
 \operatorname{CH}^{n+1}(J\times_K X)$ such that the composition $$
 \begin{CD}
 \operatorname{Ab}^{n+1}(X_{\bar K})@> w_{Z_{\bar K}}>>
 \operatorname{A}^{n+1}(X_{\bar K}) @>\phi^{n+1}_{X_{\bar K}} >>
 \operatorname{Ab}^{n+1}(X_{\bar K})\end{CD}
 $$
 is induced by the  $K$-morphism  $r\cdot \operatorname{Id}_{J}:J\to J$ for some
 non-zero integer $r$, and the maps  $w_{Z_{\bar K}}$ and $\phi^{n+1}_{X_{\bar
   K}}$ are
 $\operatorname{Gal}(K)$-equivariant. 
 Let $\lambda:\widehat J\to J$ be a polarization and let  $\Theta$ be the $(\dim J
 -1)$-fold self-intersection of the class of the 
 polarization on $J$.
 The correspondence $Z\circ \Theta \circ \Gamma_\lambda$, defined over $K$, 
 induces a  morphism of
 $\operatorname{Gal}(K)$-representations
 \begin{equation}\label{E:212'} 
 \begin{CD}
 H^1(\widehat J_{\bar K},\mathbb Q_\ell)@>Z_*\circ \Theta_*\circ \Gamma_{\lambda *}>>
 H^{2n+1}(X_{\bar
  K},\mathbb Q_\ell(n)).  
 \end{CD}
 \end{equation}
 
Finally,  if we assume further that $(J^{2n+1}_a(X_{\mathbb C}), AJ)\cong
 (\operatorname{Ab}^{n+1}(X_{\mathbb C}),\phi^{n+1}_{X_{\mathbb C}})$, then 
 \eqref{E:HdgMor} and the
 comparison isomorphisms \eqref{E:CompIsom} imply that the map  \eqref{E:212'} is
 an
 inclusion, with image \linebreak $\coniveau ^n H^{2n+1}(X_{\bar K},\mathbb Q_\ell(n))$.  
The comparison isomorphisms \eqref{E:CompIsom} also show that
 $H^{2n+1}(X_{\mathbb C},\mathbb Q)$ is of geometric coniveau $n$ if and only if 
 $H^{2n+1}(X_{\overline K},\mathbb Q_\ell)$ is of geometric coniveau $n$\,;
 therefore
 \eqref{E:212'} is an isomorphism under the hypotheses given in the theorem.  
\end{proof}

\begin{cor}\label{C:cycles3} 
 Suppose that    $X$ is a smooth 
 projective  variety over a field $K\subseteq
 \mathbb C$.  
 The  abelian variety  $J^3_{a}(X_{\mathbb C})$
 has a distinguished  model $ J$ over $K$.  There is a correspondence $\gamma$ on
 $\widehat J\times_K X$ such that for each prime number
 $\ell$ the correspondence induces  an inclusion of 
 $\operatorname{Gal}(K)$-representations
 \begin{equation}\label{E:Ccycles3}
 H^1( \widehat J_{\bar K},\rat_\ell)\stackrel{\gamma_*}{\hookrightarrow}  H^3(X_{\bar
  K},\rat_\ell(1)),
 \end{equation}
 with image $\coniveau^1H^3(X_{\bar K},\rat_\ell(1))$.   
 Consequently, if $H^3(X_{\mathbb C},\mathbb Q)$ is of geometric coniveau $1$ (e.g., if $X_{\mathbb C}$ is a uni-ruled threefold),
 then \eqref{E:Ccycles3} is an isomorphism.  
\end{cor}

\begin{proof}
 This follows from Theorem \ref{T:cycles3'} and \cite{murre83} (see Theorem
 \ref{T:Murre}).  That a uni-ruled threefold $X_\mathbb{C}$ has $H^3(X_{\mathbb C},\mathbb Q)$ of geometric coniveau $1$ follows, via a decomposition of the diagonal argument \cite{BlSr83}, from the fact that $\operatorname{CH}_0(X_\mathbb{C})$ is supported on a surface.
\end{proof}

\begin{rem}
 Unlike the approach taken to prove Theorem \ref{T:cycles'}, the approach to
 proving Theorem \ref{T:cycles3'} using Murre's results \cite{murre83} does not
 seem to provide a splitting for the inclusion \eqref{E:Ccycles3}. 
 However, since we are taking $2n+1=3$,  the
 arguments in the proof of Theorem \ref{T:cycles'}  do provide a splitting for
 \eqref{E:Ccycles3},  induced by an algebraic correspondence if $K$ is finitely generated.
 Indeed, let $J'$ be the abelian variety over $K$ and  $\gamma'$
 be the correspondence provided by Theorem
 \ref{T:cycles'}. The inclusion \eqref{E:Ccycles3} and the inclusion $\gamma_*' :
 H^1(J'_{\bar K},\rat_\ell) \hookrightarrow H^{3}(X_{\bar K},\rat_\ell(1))$ have the 
 same image, namely $\coniveau^1H^{3}(X_{\bar K},\rat_\ell(1))$. Therefore
 $H^1(J_{\bar K},\rat_\ell)$ and $H^1(J'_{\bar K},\rat_\ell)$ are isomorphic as
 $\operatorname{Gal}(K)$-representations.  The splitting of $\gamma'_*$ from Theorem \ref{T:cycles'} then provides a splitting of \eqref{E:Ccycles3}.
 If moreover $K$ is finitely generated, then $J$ and $J'$ are $K$-isogenous, in which case composing the the graph of such an isogeny with the correspondence from Theorem \ref{T:cycles'}
 that gives   the splitting of $\gamma_*'$  yields a correspondence splitting   \eqref{E:Ccycles3}.
\end{rem}

\begin{rem}\label{R:corr}
 Comparing Corollary \ref{C:cycles3} to Corollary \ref{C:Murre}, we see that by
 assuming that $K\subseteq \mathbb C$ and by using the functoriality of the Abel--Jacobi map with respect to the action of correspondences, we are able to obtain the inclusion of
 $\operatorname{Gal}(K)$-representations  via a correspondence, and to characterize its image in terms of the geometric coniveau filtration.
\end{rem}

\section{Complements on phantoms and algebraic representatives}\label{S:Complements}

In this section we discuss specialization of phantom abelian varieties and  algebraic
representatives.  This has particular relevance to the question of phantoms from the perspective of  Honda--Tate theory.  Motivated by Theorem
\ref{T:cycles3'}, we also establish some results on algebraic
representatives for higher codimension cycles.

\subsection{Phantoms, algebraic representatives, and specializations} \label{S:Specialization} 

\begin{lem}
\label{lem:goodred}
Let $S = \spec R$ be a discrete valuation ring with generic point $\eta = \spec
K$ and special point $0 = \spec \kappa$, and let $\ell$ be a 
prime number 
invertible in $\kappa$.   Let $ X/S$ be a smooth,
projective scheme, and let $ A/\eta$ be an abelian variety equipped with
a $\operatorname{Gal}(K)$-equivariant  inclusion
\begin{equation}
\label{eqphantomdiag}
\xymatrix{
H^1( A_{\bar \eta},\rat_\ell) \ar@{^(->}[r] &
H^{2n+1}(X_{\bar\eta},\rat_\ell(n)).
}
\end{equation}

\begin{alphabetize}

\item Then $ A$ extends to an abelian scheme $ A/S$.

\item Suppose the inclusion \eqref{eqphantomdiag} 
   has image $\coniveau^n
  H^{2n+1}(X_{\bar\eta},\rat_\ell)$ and   that $\dim \coniveau^n H^{2n+1}(X_{\bar
    0},\rat_\ell) = \dim \coniveau^n
  H^{2n+1}(X_{\bar\eta},\rat_\ell)$.  Then  specialization   induces an inclusion 
\begin{equation}\label{E:PhantD2}
    \xymatrix{
H^1( A_{\bar 0},\rat_\ell) \ar@{^(->}[r] &
H^{2n+1}(X_{\bar 0},\rat_\ell(n)),
}
 \end{equation}
with image $\coniveau^n H^{2n+1}(X_{\bar
    0},\rat_\ell)$.  Moreover, if  \eqref{eqphantomdiag} 
 is realized  by a correspondence  on
  $A \times_\eta  X_\eta$, then \eqref{E:PhantD2} is realized by the specialization of the correspondence on $A_0\times_0X_0$.

  \item If $ A_\eta$ is a phantom for $ X_\eta$ (i.e., \eqref{eqphantomdiag} is an isomorphism), then $
  A_0$ is a phantom for $ X_0$ (i.e., \eqref{E:PhantD2} is an isomorphism).

\end{alphabetize}
\end{lem}

\begin{proof} (a) 
Since $X_\eta$ has good reduction, it follows that $H^{2n+1}(X_{\bar\eta},\rat_\ell)$
is unramified as a representation of $\gal(K)$, and thus so is
$H^1(A_{\bar\eta},\rat_\ell)$.  The N\'eron--Ogg--Shafarevich criterion
shows $A$ extends to $S$.  

(b) This  follows from  the canonical
isomorphism $H^r(Y_{\bar\eta},\rat_\ell) \iso H^r(Y_{\bar
  0},\rat_\ell)$ provided by proper base change for any smooth
projective $Y/S$.  More precisely, 
 since specialization is compatible with
intersection \cite[20.3.5]{fulton},
 the specialization isomorphism
gives an inclusion $\coniveau^n
H^{2n+1}(X_{\bar\eta},\rat_\ell)\subseteq \coniveau^n H^{2n+1}(X_{\bar
  0},\rat_\ell)$.  
 Thus we have a commutative diagram
 $$
    \xymatrix@R=1em{
H^1( A_{\bar \eta },\rat_\ell) \ar@{->}[r]^<>(0.5){\sim} \ar@{->}[d]^\sim &\coniveau^n H^{2n+1}(X_{\bar \eta},\rat_\ell(n)) \ar@{^(->}[r]  \ar@{^(->}[d] &
H^{2n+1}(X_{\bar \eta},\rat_\ell(n)) \ar[d]^\sim\\
H^1( A_{\bar 0 },\rat_\ell) \ar@{-->}[r] &\coniveau^n H^{2n+1}(X_{\bar 0},\rat_\ell(n)) \ar@{^(->}[r]  &
H^{2n+1}(X_{\bar 0},\rat_\ell(n)).\\
}
 $$
 A dimension count implies that the induced dashed arrow is an isomorphism.  
 Finally, since  specialization is compatible with
cup product in cohomology,
the dashed arrow is in fact induced by the specialization of the correspondence inducing \eqref{eqphantomdiag}.

(c)  follows from (b).
\end{proof}

As an immediate consequence, we find\,:

\begin{cor}  With $X/S$ as in Lemma \ref{lem:goodred}, suppose further that $\operatorname{char} (K)=\operatorname{char}(\kappa)=0$ and that $\dim \coniveau^1
H^{3}(X_{\bar\eta},\rat_\ell)=  \dim \coniveau^1
H^{3}(X_{\bar 0},\rat_\ell)$.
 Then $\ubar\Ab^2(X_{\bar\eta})_0$  (Theorem \ref{T:Desc} and Remark \ref{R:DescAlgRep})  and  $\ubar \Ab^2(X_{\bar 0})$ are isogenous abelian varieties over $\kappa$.
\end{cor}

\begin{proof}
Use Lemma \ref{lem:goodred} and Corollary \ref{C:cycles3}.
\end{proof}

\begin{rem}
If one assumes further in the corollary that $\phi^2_{X_{\bar\eta}}$ and $\phi^2_{X_{\bar 0}}$ are isomorphisms (e.g., by \cite[Th. 1(i)]{BlSr83}, if $\operatorname{CH}_0(X_{\bar \eta})$ and $\operatorname{CH}_0(X_{\bar 0})$ are supported on curves after any base-change of algebraically closed field), then one may conclude in fact that $\ubar\Ab^2(X_{\bar\eta})_0$  and 
$\ubar \Ab^2(X_{\bar 0})$ are isomorphic. 
\end{rem}

In the setting of Mazur's original question (over a number field $K$), the following corollary  shows that under the geometric coniveau hypothesis 
 there is
an abelian variety over $K$ which interpolates the isogeny classes
described by Honda--Tate theory.

\begin{cor}
Let $S$ be a Dedekind scheme with generic point $\eta$.  Let $X \to S$
be a smooth projective morphism, and let $A_\eta/\eta$ be a phantom
abelian variety for $X_\eta$ in degree $2n+1$ (i.e., \eqref{eqphantomdiag} is an isomorphism).  Then $A_\eta$ extends to an
abelian scheme $A/S$ such that, for each point $s\in S$, $
A_s$ is a phantom abelian variety for $X_s$ in degree $2n+1$.
\end{cor}

\begin{proof}
Let $A$ be the N\'eron model of $A$ over $S$.  Lemma \ref{lem:goodred}
shows that every fiber is an abelian variety (and thus $ A$ is an
abelian scheme), and that each fiber is a
phantom for the corresponding fiber of $X$.
\end{proof}

However, like the N\'eron--Severi group, the rank of the   geometric
coniveau $n$ piece  in degree $2n+1$ can vary in families\,; this behavior rules out the
existence, in general, of a relative phantom abelian scheme modeling the coniveau $n$ piece in degree $2n+1$ for  families of projective varieties.  In particular, algebraic representatives need not specialize to algebraic representatives.   The following provides such an example\,:

\begin{exa}
Let $S=\operatorname{Spec} R$ be the spectrum of a discrete valuation ring as in Lemma \ref{lem:goodred}, and let $E$ be a CM field such that $[E:\rat] = 6$
and $[\widetilde E:\widetilde{E^{(+)}}] = 8$, where $\widetilde E$ and
$\widetilde{E^{(+)}}$ are the normal closures over $\rat$ of $E$ and its
maximal totally real subfield, respectively.  Let $X/S$ be an 
abelian threefold such that $\End(X_{\bar\eta}) \iso \integ$ and
$\End(X_0) \iso \End(X_{\bar 0})\supseteq \calo_E$.  On one hand, \cite[Thm.\
1]{tankeevghc} implies that $\dim \coniveau^1
H^3(X_{\bar\eta},\rat_\ell)=6$.  On the other hand, \cite[Thm.\
2]{tankeevghc} implies that if $\operatorname{char}(\kappa) = 0$, then
$\dim \coniveau^1 H^3(X_{\bar 0},\rat_\ell) = 18$, while if
$\operatorname{char}(\kappa) >0$ a further specialization argument guarantees that $\dim \coniveau^1 H^3(X_{\bar 0},\rat_\ell) \ge
18$.  Thus there is no  abelian scheme $J/S$ so that for each $s\in S$, there is an isomorphism $H^1(J_{\bar s},\mathbb Q_\ell) \cong \coniveau^1 H^3(X_{\bar s},\rat_\ell)$.  Moreover, if $\operatorname{char}(\kappa)=0$, this gives an example where  $\underline {\operatorname{Ab}}^2(X_{\bar \eta})_0$ is not isogenous to $\underline {\operatorname{Ab}}^2(X_{\bar 0})$. 
\end{exa}

\subsection{Algebraic representatives for higher codimension cycles} \label{S:comp}
In order to employ Theorem \ref{T:cycles3'}, one must first have an algebraic representative. 
However, outside of the cases of codimension-$1$,  codimension-$2$, and codimension-$d$ ($=\dim X$) cycles, there are not many results on their existence.   Saito's result  \cite[Thm.~2.2]{hsaito} (see Proposition \ref{P:saito}) provides a criterion for establishing the existence of algebraic representatives.  In this section, we use this to provide some further criteria for when algebraic representatives exist.  
To begin, Proposition \ref{P:referee} suggested by the referee,
asserts that if a  variety is dominated by a second variety that admits  an algebraic representative, then the first variety also admits an algebraic representative.   Next, Proposition \ref{P:BSalgrep} (see also Remark \ref{R:birat}) shows that the existence of algebraic representatives for codimension-$3$ cycles is a birational invariant.

\begin{pro}\label{P:referee}
 Let $K$ be a field, $n$ be a natural number, and let  $f : Y \to  X$ be a
 dominant morphism of smooth 
 projective varieties over $K$.  
 Assume there exists an algebraic
 representative for $\operatorname{A}^{n}(Y_{\bar K})$.
 Then there exists an algebraic representative for
 $\operatorname{A}^{n}(X_{\bar
  K})$.
\end{pro}

\begin{proof}
 Let $\iota : Z \hookrightarrow Y$ be a closed subvariety of  pure dimension
 $\dim X$
 cut out by hyperplanes. The morphism $f|_Z : Z \to X$ is then generically
 finite, of degree $r$, say. Assume there exists an algebraic
 representative $\operatorname{Ab}^{n}(Y_{\bar K})$ for
 $\operatorname{A}^{n}(Y_{\bar K})$.
 By the projection formula, the induced map 
 $$(f|_Z)_*(f|_Z)^* = (f|_Z)_*\iota^*f^* : \operatorname{A}^{n}(X_{\bar K}) \to
 \operatorname{A}^{n}(X_{\bar K})$$ is multiplication by $r$. This map is
 surjective because $\operatorname{A}^{n}(X_{\bar K})$ is a divisible group. It
 follows that the map $$(f|_Z)_*\iota^* :   \operatorname{A}^{n}(Y_{\bar K})
 \to
 \operatorname{A}^{n}(X_{\bar K})$$ is surjective. Consider then a surjective
 regular homomorphism $\operatorname{A}^{n}(X_{\bar K}) \to A(\bar K)$ to an
 abelian variety $A$ defined over $\bar K$. Then the composition 
 $$
 \begin{CD}
 \operatorname{A}^{n}(Y_{\bar K}) @>(f|_Z)_*\iota^*>>
 \operatorname{A}^{n}(X_{\bar
  K})@>>>
 A(\bar K)   
 \end{CD}
 $$
 is a surjective regular homomorphism. By the universal property of the
 algebraic representative $\operatorname{Ab}^{n}(Y_{\bar K})$ for
 $\operatorname{A}^{n}(Y_{\bar K})$, we find that $\dim A \leq \dim
 \operatorname{Ab}^{n}(Y_{\bar K})$. By \cite[Thm.~2.2]{hsaito} (see Proposition \ref{P:saito}), it follows that $\operatorname{A}^{n}(X_{\bar K})$
 admits an algebraic representative.
\end{proof}

\begin{pro}\label{P:BSalgrep} 
 Let $K\subseteq \Omega$ be a sub-field of a universal domain $\Omega$
 and let $X$ and $Y$ be smooth 
 projective varieties over $K$ with a correspondence $\Gamma \in
 \operatorname{CH}^d(Y\times X)$ such that
 $(\Gamma_\Omega)_* : \operatorname{CH}_0(Y_\Omega)_\rat \rightarrow
 \operatorname{CH}_0(X_\Omega)_\rat$ is surjective.
 Assume that $\operatorname{A}^3(Y_{\bar K}) $ admits an algebraic representative
 (e.g. if
 $\dim Y \leq 3$). Then
 $\operatorname{A}^3(X_{\bar K}) $ admits an algebraic
 representative.
\end{pro}
\begin{proof} The assumption that $\operatorname{CH}_0(X_\Omega)$ is spanned by
 $\operatorname{CH}_0(Y_\Omega)$, via the action of the correspondence $\Gamma$,
 implies by
 a decomposition of the diagonal argument \cite[Prop. 1]{BlSr83} (see also
 \cite[Prop. 3.5]{vialab} for the factorization assertion below) the existence of
 a
 positive integer $N$ such that
 $N \Delta_{ X} = {Z}_1 + {Z}_2$ in $\operatorname{CH}^d(X \times X)$,
 with $Z_1$ supported on  $D \times X$  for some divisor $D \subset
 X$, and $ Z_2$, when seen as a correspondence
 from $ X$ to $X$, factoring as $\Gamma \circ \Gamma'$ for some correspondence
 $\Gamma' \in \operatorname{CH}_d(X \times Y)$. Let us denote
 $\widetilde D$ an alteration of $D$ and  $\iota :
 \widetilde D \to D \to X$ the natural morphism.

 Consider a surjective regular
 homomorphism 
 $\operatorname{A}^3(X_{\bar K}) \twoheadrightarrow A(\bar K)$. We see from the divisibility of
 $\operatorname{A}^3(X_{\bar K})$ and from the decomposition of the diagonal that
 $$\operatorname{A}^3(X_{\bar K}) = (N\Delta_X)^* \operatorname{A}^3(X_{\bar K})
 = (\Gamma')^*\operatorname{A}^3(Y_{\bar K}) +
 \iota_*\operatorname{A}^2(\widetilde D_{\bar K}).$$ 
 It follows
 that the regular homomorphism 
 $$
 \begin{CD}
 \operatorname{A}^{3}(Y_{\bar K}) \oplus \operatorname{A}^2(\widetilde D_{\bar
  K}) @>(\Gamma')^*
 \oplus \iota_*>>
 \operatorname{A}^{3}(X_{\bar
  K})@>>>
 A(\bar K)   
 \end{CD}
 $$ is surjective. Since both $\operatorname{A}^{3}(Y_{\bar K})$ and
 $\operatorname{A}^2(\widetilde D_{\bar K})$ admit an algebraic representative
 (by
 assumption and by Murre's Theorem \ref{T:Murre}, respectively), we get by the
 criterion of \cite[Thm.~2.2]{hsaito} (see Proposition \ref{P:saito}) that there exists an integer
 $M$ independent of the surjective regular homomorphism
 $\operatorname{A}^3(X_{\bar K}) \twoheadrightarrow A(\bar K)$ such that $\dim A
 \leq M$. We conclude, by the same criterion, that $\operatorname{A}^3(X_{\bar
  K})$ has an algebraic representative.
\end{proof}

\begin{rem}\label{R:birat}
 The existence of an
 algebraic representative for codimension-$3$ algebraically trivial cycles is a
 birational invariant of
 smooth projective varieties over $\bar K$. Indeed, let $f :  Y \dashrightarrow  X$
 be a dominant map of smooth projective varieties over $\bar K$. Then $f_* :
 \operatorname{CH}_0(Y_\Omega) \rightarrow \operatorname{CH}_0(X_\Omega)$ is
 surjective, and it follows from Proposition \ref{P:BSalgrep}  that, if
 $\operatorname{A}^3(Y)$ has an algebraic representative, then
 $\operatorname{A}^3(X)$ has an algebraic representative. 
\end{rem}

\begin{rem}
 Proposition \ref{P:BSalgrep} can be generalized by repeated use of the
 decomposition of the diagonal argument. Specifically, one can show that if $X$
 is a smooth projective variety over $K$ such that
 $\operatorname{CH}_0(X_\Omega)_\rat,\ldots, \operatorname{CH}_l(X_\Omega)_\rat$ are
 spanned by the $\operatorname{CH}_0$ of threefolds 
 via the action of
 correspondences, then $\operatorname{A}^0(X_{\bar K}), \ldots,
 \operatorname{A}^{l+3}(X_\bar{K})$ admit  algebraic
 representatives. 
\end{rem}


\bibliographystyle{hamsalpha}
\bibliography{DCG}

\end{document}